\documentclass[smallextended,referee,envcountsect,]{svjour3}
\smartqed
\usepackage{graphicx}

\usepackage{fix-cm}

\providecommand{\norm}[1]{\lVert#1\rVert}
\usepackage{amssymb}
\usepackage{amsmath}
\usepackage{todonotes}
\usepackage{mathtools}
\newtheorem{assumption}{Assumption}
\usepackage[colorlinks=true, linkcolor=blue, citecolor=red, urlcolor=magenta]{hyperref}
\usepackage{lineno}
\usepackage{adjustbox}

\providecommand{\norm}[1]{\lVert#1\rVert}
\def\E{\mathbb{E}}
\def\N{\mathbb{N}}

\def\dist{\mathop{\mathrm{dist}}}

\usepackage[linesnumbered,ruled,vlined]{algorithm2e}
\usepackage{subfigure}
\usepackage{booktabs}
\usepackage{cleveref}
\crefname{assumption}{Assumption}{Assumptions} 
\date{Received: date / Accepted: date}
\usepackage{hyperref}

\usepackage[indLines=false,noEnd=false]{algpseudocodex}
\usepackage[kerning, spacing]{microtype}
\usepackage{multirow}

\usepackage{float}
\floatstyle{ruled}
\floatname{algorithm}{Algorithm}
\newfloat{algorithm}{tbp!}{loa}
\numberwithin{algorithm}{section}
\usepackage{comment}

\begin{document}
\title{Monotone and nonmonotone linearized block coordinate descent
methods for nonsmooth composite optimization problems}

\titlerunning{Monotone and nonmonotone linearized block coordinate descent method}        

\author{Yassine Nabou$^*$ \and 
 Lahcen El Bourkhissi$^\dagger$ \and
 Sebastian U. Stich$^\ddagger$ \and
 Tuomo Valkonen$^\S$}
\institute{$^*$Department of Mathematics and Statistics, University of Helsinki, Finland. \texttt{yassine.nabou@helsinki.fi}.\\ 
$^\dagger$Automatic Control and Systems Engineering, National University for Science and Technology Politehnica Bucharest, Romania, \texttt{lel@stud.acs.upb.ro}.\\
$^\ddagger$CISPA Helmholtz Center for Information Security, Germany, 
\texttt{stich@cispa.de}.\\
$^\S$Research Center in Mathematical Modeling and Optimization (MODEMAT), Quito, Ecuador \emph{and} Department of Mathematics and Statistics, University of Helsinki, Finland, \texttt{tuomo.valkonen@iki.fi}. 
}

\maketitle

\begin{abstract}

 In this paper, we introduce both monotone and nonmonotone variants of LiBCoD, a \textbf{Li}nearized \textbf{B}lock \textbf{Co}ordinate \textbf{D}escent method for solving composite optimization problems. At each iteration, a random block is selected, and the smooth components of the objective are linearized along the chosen block in a Gauss-Newton approach. For the monotone variant, we establish a global sublinear convergence rate to a stationary point under the assumption of bounded iterates. For the nonmonotone variant, we derive a global sublinear convergence rate without requiring global Lipschitz continuity or bounded iterates. Preliminary numerical experiments highlight the promising performance of the proposed approach.

\end{abstract}


\section{Introduction and motivation}
\label{submission}

Effective optimization methods are fundamental for the realization of machine learning techniques, enabling advances in tasks such as training complex models and improving algorithm efficiency \cite{bottou2018optimization}. Beyond machine learning, optimization plays a central role in diverse fields, such as economics \cite{goetzmann2014modern} and engineering \cite{RawMay:17}, where it drives innovation and problem-solving.
For optimization algorithms to be effective, to achieve maximal performance, it is crucial that they exploit the intricate structures of each problem in question.
A particularly important class of structured optimization problems involves minimizing the sum of a convex function and a composition of a convex function with a smooth map, i.e., 
\begin{equation}  
    \label{eq:problem}  
    \min_{x \in \text{dom}\,\varphi} \varphi(x): = f(x) + h(F(x)) + \sum_{i=1}^{n_{\text{block}}} g_i(x^i), 
\end{equation}  
where \( F(x): = (F_1(x), \cdots, F_m(x)) \) and \( f \) are nonlinear differentiable functions defined on \(\text{dom}\,\varphi\). The functions \( g_i \) are proper, lower semicontinuous (lsc), and convex.
They may be nonsmooth, while \( h \) is a smooth and convex function. Here, \( n = \sum_{i=1}^{n_{\text{block}}} n_i \), \( x^i \in \mathbb{R}^{n_i} \), and \( g(x) := \sum_{i=1}^{n_{\text{block}}} g_i(x^i) \).\footnote{Note that, from the point of view of general theory, $f$ is superfluous, and can be subsumed by $h \circ F$. It is, however, practical to explicitly include it for our application examples.}
The problem \eqref{eq:problem} encompasses a wide range of applications across machine learning and signal processing. For instance, in nonlinear regression problems \cite{dutter1981numerical}, \( h \) represents a loss function, \( F \) denotes the model being trained, e.g., a neural network, and \( g \) acts as a regularizer. Similarly, it is applicable to risk parity portfolio selection \cite{maillard2010properties}, which addresses robust financial decision-making, and robust phase retrieval \cite{duchi2019solving}, a challenge involving the recovery of phase information from magnitude measurements. Additionally, it arises in PDE-constrained inverse problems \cite{roosta2014stochastic}, where the goal is to recover unknown parameters governed by partial differential equations. In this paper, we propose a block coordinate descent approach to address problem \eqref{eq:problem}, which leverages the structure of the problem through a Gauss-Newton setting, allowing for the use of a more approximate model. Existing methods exploit the problem's structure but not within a block coordinate descent framework, motivating our approach.
\paragraph{Related Work}
The gradient method is a widely used optimization technique for solving smooth and differentiable problems by iteratively updating variables to achieve optimality. Its iterative update rule arises from finding the minimizer of the quadratic surrogate model
\[
    f(x) \leq  f(x_k) + \nabla f(x_k)^T (x - x_k) + \frac{1}{2\alpha_k}\|x - x_k\|^2.
\]
This gives rise to the update rule
\[
    x_{k+1} = x_k - \alpha_k \nabla f(x_k).
\]
The step size \( \alpha_k \) is critical for convergence and can be chosen using methods like backtracking line search or a fixed schedule. This iterative process is continued until the gradient norm \( \|\nabla f(x_k)\| \) is sufficiently small, indicating proximity to a stationary point. The gradient method is widely recognized in the optimization community and beyond for its simplicity, low computational cost per iteration, and solid theoretical guarantees across a broad spectrum of problems, including smooth convex and smooth nonconvex formulations \cite{nesterov2018lectures}. 

When the problem takes the form \eqref{eq:problem} with nonsmooth $g_i$, the gradient method can no longer be applied. The proximal gradient method was proposed in \cite{nesterov2013gradient} to handle such problems.
For a desired accuracy \(\epsilon > 0\), the proximal gradient method requires at most \(\mathcal{O}\left(\frac{L_h}{\epsilon^2}\right)\) iterations to achieve an \(\epsilon\)-stationary point of problem \eqref{eq:problem}.
However, a significant limitation of gradient-based methods is that they treat the composite structure \( h(F(x)) \) as a single entity, ignoring its separable nature. This can lead to suboptimal approximations and slower convergence in some cases.

To address this limitation, prox-linear/Gauss-Newton type methods were developed to solve nonlinear least squares problem, where \( h(\cdot) = \frac{1}{2}\|\cdot\|^2 \) \cite{nocedal1999numerical}. These methods are well-suited for optimization problems with a composite structure \( h(F(x)) \), as they preserve and exploit the structural coupling between \( F(x) \) and \( h \), e.g, \cite{nesterov2007modified,cartis2011evaluation,salzo2012convergence,jauhiainen2019gaussnewton,marumo2024accelerated}. Instead of fully linearizing the smooth component \( \ell(x) \), Gauss-Newton methods linearize only \( F(x) \), leaving \( h(\cdot) \) unchanged, and add a quadratic regularization term to control the approximation error. This approach results in a more accurate approximation of the objective function compared to the linearized model used in gradient methods.
Recent work in \cite{marumo2024accelerated} have demonstrated that for problems of the form \eqref{eq:problem}, a Gauss-Newton-type method can achieve an $\epsilon$-stationary point in $\mathcal{O}\left(\frac{\sqrt{L_h}}{\epsilon^2}\right)$ iterations, offering an improvement over gradient-based methods, especially when the Lipschitz constant $L_h$ is large.

The aforementioned methods are generally monotone, meaning that they either require or implicitly ensure a monotonic decrease in the objective function. To maintain this monotonicity, it is typically necessary to assume that the smooth part of the objective function is Lipschitz continuous over the entire space or to assume that the generated iterates remain bounded. However, these assumptions can be relaxed (see, e.g., \cite{kanzow2022convergence,de2023proximal,kanzow2024convergence}).  
For instance, \cite{kanzow2024convergence} proposed a nonmonotone method for minimizing a \textit{simple composite problem} of the form \( f(x) + g(x) \), where \( f \) is locally smooth, and the nonsmooth function \( g \) is bounded below by an affine function. Notably, this method achieves convergence rates comparable to those of monotone approaches. However, to  our knowledge, a nonmonotone method has not yet been specifically developed for composite problems of the form~\eqref{eq:problem} within a Gauss-Newton framework.  

Furthermore, when the decision variable \( x \) has a high dimensionality, which is often the case in machine learning and related fields, computing the full gradient can become computationally expensive. To mitigate this, (block) coordinate methods have been introduced. These methods focus on computing the gradient of the smooth component with respect to a single coordinate or block of coordinates at each iteration. For problems of the form \( f(x) + g(x) \), several coordinate-based methods have been proposed, e.g., \cite{nesterov2012efficiency,richtarik2014iteration,patrascu2015efficient,wright2015coordinate,birgin2022block,deng2020efficiency}.  
For instance, Nesterov \cite{nesterov2012efficiency} established complexity results for coordinate gradient descent methods applied to smooth and convex optimization problems. Building on this work, paper \cite{richtarik2014iteration} extended these results to convex problems with simple composite objective functions. For nonconvex simple composite problems, papers \cite{patrascu2015efficient,wright2015coordinate,birgin2022block,deng2020efficiency} provided convergence analyses, offering valuable insights into the behavior of coordinate descent in nonconvex settings.
While problem \eqref{eq:problem} could theoretically be solved using the standard coordinate descent techniques described earlier by treating it as a simple composite optimization problem ($f(x) + g(x)$), this approach does not fully leverage the problem's structure.

Another approach closely related to block coordinate methods is the use of randomized subspace techniques, which involve projecting the original high-dimensional space onto a lower-dimensional subspace. \cite{cartis2022randomised} propose a Randomized Subspace Gauss-Newton method for solving nonlinear least-squares optimization problems (i.e., problem~\eqref{eq:problem} with $h = \frac{1}{2}\|\cdot\|^2$, and $f = g = 0$). This method employs a sketched Jacobian of the residual ($F(\cdot)$) and solves a reduced linear least-squares problem at each iteration. By randomizing the subspace for minimization, the approach benefits from dimensionality reduction facilitated by results akin to the Johnson-Lindenstrauss (JL) Lemma. These results ensure the problem's dimensionality is reduced while retaining critical structural information. Furthermore, the authors establish a global sublinear rate with high probability. However, the applicability of these results is often constrained by assumptions such as global Lipschitz continuity and smoothness of the optimization problem, which are specific to smooth settings.

\paragraph{Contributions} In this paper, we introduce a novel approach that combines the strengths of the prox-linear/Gauss-Newton setting, which leverages the problem's structure, with block coordinate descent, which reduces both per-iteration computational complexity and memory requirements. This integration results in the development of the monotone LiBCoD method. Additionally, by incorporating the benefits of nonmonotone schemes, we extend our approach to propose a second method, referred to as nonmonotone LiBCoD. Specifically, our main contributions are
\begin{itemize} 
    \item[(i)] We propose two methods in which, at each iteration, we linearize the smooth part of the objective function with respect to randomly chosen block coordinates in a Gauss-Newton type approach and add a dynamic regularization term. The resulting algorithms require solving a strongly convex subproblem at each iteration, which makes them easy to solve.  
    \item[(ii)] We provide global asymptotic convergence, proving that the iterates of both methods converge, in expectation, to a stationary point of problem \eqref{eq:problem}. Additionally, our methods guarantee convergence to an \(\epsilon\)-stationary point of problem \eqref{eq:problem} in at most \(\mathcal{O}(1/\epsilon^2)\) iterations.  
    \item[(iii)] We test the performance of the proposed methods and compare it with several existing baseline methods from the literature. 
\end{itemize}  
To our knowledge, this is the first work to propose and analyze these methods, offering an efficient solution to composite optimization problems that exhibit the structural characteristics of \ref{eq:problem}.

\paragraph{Structure of the paper} This paper is organized as follows. Section \ref{sec2} introduces the notation used throughout the paper. In Section \ref{sec3}, we present the monotone variant of our method along with its convergence analysis, followed by the nonmonotone variant in Section \ref{sec4}. Section \ref{sec5} addresses nonconvex optimization with nonlinear equality constraints, utilizing the quadratic penalty method. Finally, in Section \ref{sec6}, we provide a numerical comparison of our method with existing algorithms.
\section{Notations and preliminaries} \label{sec2}
In what follows, we use $\| \cdot \|$ to denote the 2-norm of a vector or a matrix, respectively. Let $n = \sum_{i=1}^{n_\text{block}} n_i$. For a differentiable function $\phi: \mathbb{R}^n \to \mathbb{R}$, we denote its gradient with respect to the $i^{\text{th}}$ block at a point $x$ as $\nabla_i \phi(x) \in \mathbb{R}^{n_i}$, for any $i \in \{1, \dots, n_\text{block}\}$. For a differentiable vector function $F: \mathbb{R}^n \to \mathbb{R}^m$, we denote its Jacobian with respect to the $i^{\text{th}}$ block at a given point $x$ as $\nabla_i F(x) \in \mathbb{R}^{m \times n_i}$. We define the distance from a point $x$ to a set $A$ as $\dist(x, A) := \min_{a \in A} \| x - a \|$. Furthermore, for a proper, lower semicontinuous (lsc) function $g: \mathbb{R}^n \to \mathbb{R}$, we use $\partial g(x)$ to refer to its Mordukhovich (aka.~limiting) subdifferential at a point $x$, and $\partial^{\infty} g(x)$ to refer to the horizon subdifferential.
This, instead of the plain convex subdifferential, is required due to the nonconvexity of $h \circ F$ and hence the composite objective $\phi$.
For more details about the subdifferential of nonsmooth functions, we refer to \cite{rockafellar2009variational}. In the following, we say that a point \( x^* \in \mathbb{R}^n \) is a stationary point of problem~\eqref{eq:problem} if \( 0 \in \partial \varphi(x^*) \). Moreover, for \( \epsilon > 0 \), \( x^* \) is said to be an \( \epsilon \)-stationary point of problem \eqref{eq:problem} if \( \dist\left( 0, \partial \varphi(x^*) \right) \leq \epsilon \).
\section{Monotone Linearized block coordinate descent method} \label{sec3}

In this section, we present the monotone LiBCoD method for solving the composite problem \eqref{eq:problem} along with the necessary assumptions for our analysis. We begin by introducing the following assumptions:
\begin{assumption}\label{ass:1}
    \begin{enumerate} 
    \item Given a compact set $\mathcal{C}$, the gradients of $f$ and $F_j$ for $j=1:m$ are block coordinate-wise Lipschitz continuous in $\mathcal{C}$, i.e., $\forall i\in \{1,\cdots,n_{\text{block}}\},\;x,y\in \mathcal{C}$
        \begin{align*}
        \|\nabla_i f(x) - \nabla_if(y)\|\leq L^{f}_i \|x - y\|,
        \|\nabla_i F_j(x) - \nabla_i F_j(y)\|\leq L^{F_j}_i \|x - y\|.
        \end{align*}
    \item $h$ is convex, differentiable with Lipschitz gradient
        \begin{align*}
            \|\nabla h(u) - \nabla h(v)\|\leq L_{h} \|u - v\|,\;\forall \;u,v\in\mathbb{R}^{m}.
        \end{align*}
    \item Each $g_i$ is a proper, lsc, convex function.
    \item $ \forall x\in\mathcal{C}$, $u\in\mathbb{R}^m$ and $i\in\{1,\cdots,n_{\text{block}}\}$ we have
        $
             f(x) \geq f^*,\; h(u)\geq h^*,\;g_i(x) \geq g^*.
        $
    \end{enumerate}
\end{assumption}
Note that these assumptions are standard in composite and nonconvex optimization problems \cite{marumo2024accelerated,xie2021complexity}.
In the sequel, for all $i=\{1,\ldots,n_{\text{block}}\}$, $\bar{x}^i, s\in\mathbb{R}^{n_i}$ and $\bar{x}\in\mathbb{R}^n$, we denote
\begin{align*}
    &l^i_f(s;\bar{x}^i)
    :=
    f(\bar{x})+\langle\nabla_i f(\bar{x}),s-\bar{x}^i\rangle,\quad
    l^i_F \big(s;\bar{x}^i \big):=F(\bar{x})+ \nabla_i F(\bar{x}) \big(s-\bar{x}^i \big),\\
    &\quad\text{and}\quad
    \bar{\varphi}^i(s;\bar{x}) = l^i_f \big( s;\bar{x}^i \big)+ h \big( l^i_F(s;\bar{x}^i) \big) + g_i(s).
\end{align*}
We present in \cref{alg:libcod} our monotone method for solving problem  \eqref{eq:problem}.
\begin{algorithm*}[ht]
\caption{Monotone LiBCoD }
\label{alg:libcod}

\KwIn{$x_0$, $\beta_1 \geq \frac{\beta_{\min}}{2}>0$, $p_{\min} \in (0,1)$}
Set $k = 0$

\While{termination criterion not satisfied}{
  Pick a random index $i_k \in \{1,\cdots,n_{\text{block}}\}$ with probability $p_{i_k} \geq p_{\min}$\:

  \Repeat(
  ){
    \textcolor{blue}{$\varphi(x_{k+1}) \leq  \varphi(x_k) - \frac{\beta_{k+1}}{2}\|x_{k+1} - x_k\|^{2}$}
  }{
    $\beta_{k+1} := 2\beta_{k+1}$ \label{line:lin_search}

    $s_{k+1} := \arg\min\limits_{s \in \mathbb{R}^{n_{i_k}}} \resizebox{!}{2ex}{$\langle\nabla_{i_k} f(x_k),s-x_k^{i_k}\rangle + h\left(F(x_k)+\langle\nabla_{i_k} F(x_k),s-x_k^{i_k}\rangle\right) + g_{i_k}(s) + \frac{\beta_{k+1}}{2}\|s - x^{i_k}\|^2$}$ 
    
    $x^{i_k}_{k+1} := s_{k+1}$ and $x^{i}_{k+1} = x_k^i$ for $i\neq i_k$
  }

  Set $k=k+1$ and $\beta_{k+1} = \max\left\{\frac{\beta_{k}}{4},\frac{\beta_{\min}}{2}\right\}$
}
\end{algorithm*}
Note that our subproblem on Monotone LiBCoD \ref{alg:libcod} is constructed along a single block, making it significantly more efficient to solve compared to constructing it over the full block. Furthermore, the subproblem on Monotone LiBCoD \ref{alg:libcod} is strongly convex, allowing the use of the accelerated proximal gradient method \cite{beck2009fast,nesterov2018lectures}, which ensures linear convergence. In what follows, we will frequently use the following notations:
$
\Delta{x}_{k+1} := x_{k+1} - x_k .
$

\subsection{Sufficient decrease}
We start by proving that the inner process, in Monotone LiBCoD \ref{alg:libcod}, is well-defined, i.e., there exists a suitable $\beta_{k+1}$ that guarantees the decrease in Line 8 of \ref{line:lin_search}. Specifically, under \cref{ass:1}, we show that if $\beta_{k+1}$ satisfies
\begin{equation} \label{eq:beta_cond}
    \beta_{k+1} \geq L^f_{i_k} + L^F_{i_k}\sqrt{2 L_h}\sqrt{\varphi(x_k)- \varphi^*},
\end{equation}
where $\varphi^* = h^* + f^* + g^*$, then the process in Monotone LiBCoD \ref{alg:libcod}, is well defined. 
Similar choice has recently been considered in \cite{marumo2024accelerated} in the case where $f\equiv 0$ (i.e., $L^f_{i}  = 0$ for all $i$). 
\begin{lemma}
    \label{lemma1}
    Let $(x_k)_{k\geq 0}$ be generated by Monotone LiBCoD \ref{alg:libcod}. If \cref{ass:1} holds on a compact set $\mathcal{C}$ containing the iterates, and $\beta_{k+1}$ satisfies \eqref{eq:beta_cond}, then
    \begin{align}
        \label{decrease}
         \varphi(x_{k+1}) \leq  \varphi(x_{k}) - \frac{\beta_{k+1}}{2} \|x_{k+1} - x_k\|^2.
    \end{align}
\end{lemma}
\begin{proof}
    This proof follows a reasoning similar to that in \cite{marumo2024accelerated} [Lemma 3.1], adapted to the coordinate descent framework with an additional smooth function $f$.  
    Note that the subproblem's objective function
    \[
        q(s) := \bar{\varphi}^{i_k}(s;x_k)+\frac{\beta_{k+1}}{2}\|s-x^i_k\|^2,
    \]
    By the strong convexity of this objective and $s_{k+1}$ being a minimiser, i.e., $0 \in \partial q(s_{k+1})$,
    \[
        q(x_k^i)
        \ge
        q(s_{k+1}) + \frac{\beta_{k+1}}{2}\|s_{k+1}-x^i_k\|^2.
    \]
    Hence, expanding $q$,
    the fact that $\bar{\varphi}^{i_k}(x^{i_k}_k;x_k)=\varphi(x_{k}) - \sum_{i\neq i_k} g_i(x^i_{k})$ and making use of the fact that $x^i_{k+1} = x^i_k$ for $i\neq i_k$, we get
    \begin{equation} \label{initial_decrease}
        \bar{\varphi}^{i_k}(s_{k+1};x_k)\leq \varphi(x_{k}) - \sum_{i\neq i_k} g_i(x^i_{k+1}) -\beta_{k+1}\|x_{k+1}-x_k\|^2.
    \end{equation}
    Furthermore, since \( f \) has a block Lipschitz gradient with a Lipschitz constant \( L_{i_k}^f \), it follows that \cite{nesterov2018lectures}
        \begin{equation}
        \label{smooth_object}
        f(x_{k+1}) - l_f^{i_k}(s_{k+1};x_k) \leq \frac{L^f_{i_k}}{2}\|x_{k+1}-x_k\|^2.
    \end{equation}
     Moreover, since \( h \) has a Lipschitz gradient with a Lipschitz constant \( L_h \), we have
    \begin{equation}
    \label{decrease_feasib}
    \begin{split}
         h(F(x_{k+1}))& -  h(l^{i_k}_F(s_{k+1};x_k)) 
        \\
        &\leq   \langle \nabla h \left( l^{i_k}_F(s_{k+1};x_k)\right), F(x_{k+1})-l^{i_k}_F(s_{k+1};x_k) \rangle \\
        \MoveEqLeft[-1] + \frac{L_h}{2}\|F(x_{k+1})-l^{i_k}_F(s_{k+1};x_k)\|^2 
        \\
        & \leq \left\|\nabla h \left( l^{i_k}_F(s_{k+1};x_k)\right)\right\| \|F(x_{k+1})-l^{i_k}_F(s_{k+1};x_k) \|\\
        \MoveEqLeft[-1] + \frac{L_h}{2}\|F(x_{k+1})-l^{i_k}_F(s_{k+1};x_k)\|^2 
        \\
        &\overset{\text{Ass. } \ref{ass:1}}{\leq} \left\|\nabla h \left( l^{i_k}_F(s_{k+1};x_k)\right)\right\|  \frac{L^F_{i_k}}{2}\|\Delta x_{k+1}\|^2 + \frac{L_h}{2} \left(\frac{L^F_{i_k}}{2}\|\Delta x_{k+1}\|^2\right)^2. 
    \end{split}
    \end{equation}
    Using the Young's inequality $2\sqrt{a b} - a \leq b$ for any $a, b \geq 0$, and the following relation for $u =l^{i_k}_F(s_{k+1};x_k) \in \mathbb{R}^m$, with $v = u - \frac{1}{L_h} \nabla h(u)$
    \[
        h^* - h(u) \leq h(v) - h(u)
        \leq \langle \nabla h(u), v - u \rangle + \frac{L_h}{2} \norm{v - u}^2 = -\frac{1}{2L_h} \norm{\nabla h(u)}^2,
    \]
    we bound
    \begin{equation}
        \label{bound_grad_norm_squared}
        \begin{split}
            &
            \left\|\nabla h \left( l^{i_k}_F(s_{k+1};x_k)\right)\right\| - \frac{ \beta_{k+1} - L^f_{i_k}}{2L^F_{i_k}} \leq  \frac{L^F_{i_k}}{ 2(\beta_{k+1} - L^f_{i_k})}    \left\|\nabla h \left( l^{i_k}_F(s_{k+1};x_k)\right)\right\|^2
            \\
            &
            \leq
            \frac{L^F_{i_k}L_h}{ \beta_{k+1} - L^f_{i_k}} \left( h(l^{i_k}_F(s_{k+1};x_k)) - h^*  \right)
            =
            \frac{L^F_{i_k}L_h}{ \beta_{k+1} - L^f_{i_k}}
            \\
            \MoveEqLeft[-4]
            \left( \bar{\varphi}^{i_k}(s_{k+1};x_k) - h^* - g_{i_k}(s_{k+1})- f(x_{k+1}) + f(x_{k+1}) - l_f^{i_k}(s_{k+1};x_k)  \right)
            \\
            &
            \overset{\eqref{initial_decrease}, \eqref{smooth_object}}{\leq}
            \frac{L^F_{i_k} L_h}{\beta_{k+1} - L^f_{i_k}} \left( \varphi(x_k) - \varphi^* - \frac{2\beta_{k+1} - L^F_{i_k}}{2}\|\Delta x_{k+1}\|^2 \right)
            \\
            &
            \leq
            \frac{L^F_{i_k} L_h}{\beta_{k+1} - L^f_{i_k}} \left( \varphi(x_k) - \varphi^*\right) - L^F_{i_k} L_h\|\Delta x_{k+1}\|^2
            \\
            &
            \overset{\eqref{eq:beta_cond}}{\leq}\frac{\beta_{k+1}-L^f_{i_k}}{2L^F_{i_k}} - L^F_{i_k} L_h\|\Delta x_{k+1}\|^2,
        \end{split}
    \end{equation}
    where third inequality follows from the definition of $\bar{\varphi}^{i_k}$. Moreover, Using \eqref{bound_grad_norm_squared} in \eqref{decrease_feasib}, we deduce that
    \begin{equation}
        \label{to_use_next}
        \begin{split}
                    h\left(F(x_{k+1})\right) &- h\left(l^{i_k}_F(x_{k+1};x_k) \right)\\
        & \leq
        \frac{\beta_{k+1}-L^f_{i_k}}{2} \|\Delta x_{k+1}\|^2 - \left(\frac{L_h (L^F_{i_k})^2}{2} - \frac{L_h (L^{F}_{i_k})^2}{8} \right) \|\Delta x_{k+1}\|^4
        \\
        &=
        \frac{\beta_{k+1}-L^f_{i_k}}{2} \|\Delta x_{k+1}\|^2-  \frac{3 L_h (L^{F}_{i_k})^2}{8}  \|\Delta x_{k+1}\|^4.
        \end{split}
    \end{equation}
    Moreover, we have
    \begin{align*}
    \begin{split}
        \varphi(x_{k+1}) -\sum_{i\neq i_k} g_i(x^i_{k+1}) - \bar{\varphi}^{i_k}(s_{k+1};x_k)
        &=
        f(x_{k+1}) - l^{i_k}_f(s_{k+1};x_k) + h(F(x_{k+1}))\\
        \MoveEqLeft [-1] - h(l^{i_k}_F(x_{k+1};x_k)).
    \end{split}
    \end{align*}
    Using  \eqref{smooth_object} and \eqref{to_use_next} in the previous relation, it follows that
    \[
        \varphi(x_{k+1}) \leq \bar{\varphi}^{i_k}(s_{k+1};x_k) + \sum_{i\neq i_k} g_i(x^i_{k+1}) + \frac{\beta_{k+1}}{2} \|\Delta x_{k+1}\|^2.
    \]
    Finally, using \eqref{initial_decrease}, we get
    \begin{align*}
        \varphi(x_{k+1}) \leq  \varphi(x_{k}) - \frac{\beta_{k+1}}{2} \|\Delta x_{k+1}\|^2. 
    \end{align*}
    This proves our statement. \hfill\qed
\end{proof}
Note that in Monotone LiBCoD \ref{alg:libcod}, $\beta_{k+1}$ is always upper bounded by
\begin{equation}
\label{bar_gamma}
    \bar{\beta} := \sup_{k\geq 1}\beta_k \leq 2 \left(L^f+ L^F\sqrt{2L_h}\sqrt{\varphi(x_0)-\varphi^*}\right),
\end{equation}
 where $L^f:=\max_{i\in\{1, \ldots, n_{\text{block}}\}} L^f_{i}$  and  $L^F:=\max_{i\in\{1, \ldots, n_{\text{block}}\}} L^F_{i}$.
\noindent In the next lemma, we show that the subgradient of $\varphi(\cdot)$ in problem \eqref{eq:problem}, with respect to the chosen block, can be bounded by the difference between consecutive iterates.

\begin{lemma} \label{bound_subgradient}
Let \((x_k)_{k \geq 0}\) be generated by Monotone LiBCoD \ref{alg:libcod}. If \cref{ass:1} holds on a compact set to which the iterates belong, then, for all \(k \geq 0\), we have
\begin{equation}\label{bounde_gradient1}
    \begin{split}
     \dist\left(0, \partial_{i_k} \varphi (x_{k+1})  \right)
     & \leq \left( L^f_{i_k} + L^F_{i_k} \sqrt{2L_h} \sqrt{ \varphi(x_0) - \varphi^*} + \beta_{k+1}\right)  \|\Delta x_{k+1}\|
     \\
     \MoveEqLeft[-1] + \frac{L_h M^{F}_{i_k}L^{F}_{i_k}}{2} \|\Delta x_{k+1}\|^2,
    \end{split}
\end{equation}
where $\|\nabla_{i_k}F(x_k)\| \leq M^{F}_{i_k}$.
\end{lemma}
\begin{proof}
    By the optimality condition corresponding to the $s_{k+1}$ determining subproblem on Line 6 of Monotone LiBCoD \ref{alg:libcod}, we have
    \[
        -\nabla_{i_k} f(x_{k})-\nabla_{i_k} F(x_{k})^T \nabla h\left( l^{i_k}_F(s_{k+1};x_k)\right)-\beta_{k+1}(x_{k+1}-x_{k})\in \partial g_{i_k}(s_{k+1}).
    \]
    It then follows, by exploiting the definition of $\varphi$ and the properties of the derivative, that
    \begin{align*}
        &\dist\left(0,\partial_{i_k} \varphi (x_{k+1})\right)\\
        &=  \dist\left(-\nabla_{i_k} f(x_{k+1}) - \nabla_{i_k} F(x_{k+1})^T \nabla h \left( F(x_{k+1})\right) , \partial g_{i_k}(s_{k+1})\right)\\
        & \leq  \Bigg \|\nabla_{i_k} f(x_{k+1}) - \nabla_{i_k} f(x_k) + \Big(\nabla_{i_k} F(x_{k+1}) - \nabla_{i_k} F(x_{k})\Big)^T\nabla h\left(  F(x_{k+1})\right)  \\
        & -\beta_{k+1}\Delta x_{k+1} +\nabla_{i_k} F(x_{k})^T\left( \nabla h \left( F(x_{k+1}) \right)-\nabla h\left( l^{i_k}_F(s_{k+1};x_k)\right)\right) \Bigg\|.
    \end{align*}
    Using the triangle inequality, we further get
    \begin{equation}
        \label{eq:lem2-bndsub}
        \begin{split}
            &
            \dist\left(0, \partial_{i_k} \varphi (x_{k+1}) \right) \leq  \|\nabla_{i_k} f(x_{k+1})-\nabla_{i_k} f(x_k)\|
            \\
            \MoveEqLeft[-1] + \left\|\nabla h\left( F(x_{k+1})\right)\right\|\left\|\nabla_{i_k} F(x_{k+1}) - \nabla_{i_k} F(x_{k})\right\|+\beta_{k+1}\|\Delta x_{k+1}\|
            \\
            \MoveEqLeft[-1]
            +\|{\nabla_{i_k} F(x_{k})}\|\left\| \nabla h \left( F(x_{k+1}) \right)-\nabla h\left( l^{i_k}_F(s_{k+1};x_k)\right)\right\|
            \\
            &
            {\overset{{\text{Ass. } \ref{ass:1}}}{\leq}} \left(L^f_{i_k}+\left\|\nabla h\left( F(x_{k+1})\right)\right\| L^F_{i_k} +\beta_{k+1}\right)\|\Delta x_{k+1}\|
            + \frac{L_h M^{F}_{i_k}L^{F}_{i_k}}{2} \|\Delta x_{k+1}\|^2
            \\
            &\leq \left( L^f_{i_k} + L^F_{i_k} \sqrt{2L_h \left(h(F(x_k)) \!-\! h^*\right)} + \beta_{k+1}\right) \|\Delta x_{k+1}\|  + \frac{L_h M^{F}_{i_k}L^{F}_{i_k}}{2} \|\Delta x_{k+1}\|^2
            \\
            &
            \leq  \left( L^f_{i_k} + L^F_{i_k} \sqrt{2L_h \left(\varphi(x_k) - \varphi^*\right)} + \beta_{k+1}\right)  \|\Delta x_{k+1}\| + \frac{L_h M^{F}_{i_k}L^{F}_{i_k}}{2} \|\Delta x_{k+1}\|^2
            \\
            &
            \leq  \left( L^f_{i_k} + L^F_{i_k} \sqrt{2L_h \left(\varphi(x_0) - \varphi^*\right)} + \beta_{k+1}\right)  \|\Delta x_{k+1}\| + \frac{L_h M^{F}_{i_k}L^{F}_{i_k}}{2} \|\Delta x_{k+1}\|^2
             \\
            &
            \leq  \left( L^f + L^F \sqrt{2L_h \left(\varphi(x_0) - \varphi^*\right)} + \bar{\beta}\right)  \|\Delta x_{k+1}\| + \frac{L_h M^{F}L^{F}}{2} \|\Delta x_{k+1}\|^2,
        \end{split}
    \end{equation}
        where the third inequality follows from the following relation for $u= F(x_{k}) \in \mathbb{R}^m$, and $v = u - \frac{1}{L_h} \nabla h(u)$
    \[
        h^* - h(u) \leq h(v) - h(u)
        \leq \langle \nabla h(u), v - u \rangle + \frac{L_ h}{2} \norm{v - u}^2 = -\frac{1}{2L_h} \norm{\nabla h(u)}^2,
    \]
    and $M^F:=\max_{i\in\{1, \ldots, n_{\text{block}}\}} M^F_{i}$, and the fourth inequality follows from the fact that \( \varphi(x) = f(x) + h(F(x)) + g(x) \) and \( \varphi^* = h^* + f^* + g^* \), where \( h^*, g^*, \) and \( f^* \) are the lower bounds of \( h(x) \), \( g(x) \), and \( f(x) \), respectively, as defined in \cref{ass:1}.
    This proves our claim. \hfill\qed
\end{proof}
\subsection{Complexity of monotone LiBCoD}
In this section, we analyze the total computational complexity of Monotone LiBCoD \ref{alg:libcod}, expressed in terms of Jacobian evaluations (also referred to as iteration complexity).
Before deriving the iteration complexity, we first establish the asymptotic convergence of Monotone LiBCoD \ref{alg:libcod} in expectation.
\begin{proposition}[Asymptotic convergence]
    \label{propos:asymtotic:monotone}
    Let $(x_k)_{k\geq 0}$ be generated by Monotone LiBCoD \ref{alg:libcod}. If \cref{ass:1} holds on a compact set to which the iterates belong, then we have 
    \[
        \lim_{k\to \infty} \E \left[\dist(0, \partial\varphi(x_{k})) \right] = 0,
    \]
\end{proposition}
\begin{proof}
     Let $k\geq0$. From Monotone LiBCoD \ref{alg:libcod}, we have
    \begin{align*}
        \frac{\beta_{\min}}{2}\|x_{k+1} - x_k\|^{2}\leq \frac{\beta_{k+1}}{2}\|x_{k+1} - x_k\|^{2}   \leq  \varphi(x_k) - \varphi(x_{k+1}).
    \end{align*}
    Taking the expectation, we obtain
    \begin{align*}
        \frac{\beta_{\min}}{2} \E\left[\|x_{k+1} - x_k\|^{2}\right]  \leq  \E\left[\varphi(x_k)\right] - \E[\varphi(x_{k+1})].
    \end{align*}
    Summing the above inequality from $k=0$ to $K\geq 0$, we get
    \begin{align*}
        \frac{\beta_{\min}}{2 }\sum_{k=0}^{K}\E\left[\|x_{k+1} - x_k\|^{2}\right]  &\leq  \sum_{k=0}^{K}\left(\E\left[\varphi(x_k)\right] - \E[\varphi(x_{k+1})]\right) \\
        & = \E\left[\varphi(x_0)\right] - \E[\varphi(x_{K+1})]\\
        & \leq \varphi(x_0) - \varphi^*.
        \end{align*}
    Taking the limit $K \to \infty$, we get
    \begin{align*}
        \frac{\beta_{\min}}{2 }\sum_{k=0}^{\infty}\E\left[\|x_{k+1} - x_k\|^{2}\right]  \leq \varphi(x_0) - \varphi^* < \infty.
    \end{align*}
    Since $\beta_{\min} > 0$, it follows that
    \begin{equation}\label{zero_limit}
        \lim_{k\to\infty}{\E\left[\|x_{k+1} - x_{k}\|^2\right]}=0 \quad \text{ and } \quad \lim_{k\to\infty}{\E\left[\|x_{k+1} - x_{k}\|\right]}=0, 
    \end{equation}
    where the limit on the right-hand side of \eqref{zero_limit} follows from Jensen's inequality. For simplicity, we denote $\E_k[\cdot] =  \E[\cdot|x_k]$.  On the other, using \ref{bound_subgradient}, we deduce
      \allowdisplaybreaks
    \begin{equation}
        \label{subgrad_bound_asymp}
        \begin{split}
            &\dist(0, \partial\varphi(x_{k+1}))
            \leq
            \sum_{i=1}^{n_{\text{block}}}{\dist \left(0, \partial_i\varphi(x_{k+1})\right)}
            \\
            &
            \leq
            \frac{1}{p_{\min}}\sum_{i=1}^{n_{\text{block}}}{p_i \; \dist \left(0, \partial_{i_k}\varphi(x_{k+1})\right)}
            \\
            &
            = \frac1p_{\min}  \E_{k}\left[\dist \left(0, \partial_{i_k}\varphi(x_{k+1})\right)\right]
            \\
            &
            \overset{\eqref{bounde_gradient1}}{\leq }
            \frac1p_{\min}\E_{k}\left[  \left( L^f_{i_k} + L^F_{i_k} \sqrt{2L_h} \sqrt{ \left(\varphi(x_0) - \varphi^*\right)} + \beta_{k+1}\right)  \|\Delta x_{k+1}\| \right.\\
            \MoveEqLeft[-1] \left. + \frac{L_h M^{F}_{i_k}L^{F}_{i_k}}{2} \|\Delta x_{k+1}\|^2 \right]
            \\
            &
            \overset{\eqref{bar_gamma}}{\leq}
            \frac{1}{p_{\min}} \left( L^f + L^F \sqrt{2L_h} \sqrt{ \left(\varphi(x_0) - \varphi^*\right)} + \bar{\beta} \right)  \E_{k}\left[\|\Delta x_{k+1}\|\right]\\
             \MoveEqLeft[-1] + \frac{L_h M^F L^{F}}{2p_{\min}} \E_{k}\left[\|\Delta x_{k+1}\|^2 \right].
        \end{split}
    \end{equation}
  Taking expectation in \eqref{subgrad_bound_asymp}, we obtain
    \begin{align*}
            \E\left[\dist(0, \partial\varphi(x_{k+1}))\right] 
            &\leq
            \frac{1}{p_{\min}} \left( L^f \!+\! L^F \sqrt{2L_h} \sqrt{ \left(\varphi(x_0) \!-\! \varphi^*\right)} + \bar{\beta} \right)  \E\left[\|\Delta x_{k+1}\|\right]\\
             \MoveEqLeft[-1] +  \frac{L_h M^F L^{F}}{2p_{\min}} \E\left[\|\Delta x_{k+1}\|^2 \right]
    \end{align*}
    Passing to the limit and using \eqref{zero_limit}, we get
      $  \lim_{k\to \infty} \E\left[\dist(0, \partial\varphi(x_{k})) \right] = 0.$ \hfill\qed
\end{proof}
We drive next the complexity of Monotone LiBCoD \ref{alg:libcod} to an approximate stationary point of problem \eqref{eq:problem}.
\begin{theorem}[Iteration complexity]
    \label{thm:convergence:monotone}
    Let $(x_k)_{k\geq 0}$ be generated by Monotone LiBCoD \ref{alg:libcod}. If \cref{ass:1} holds on a compact set to which the iterates belong, then, for any $\epsilon>0$,  Monotone LiBCoD \ref{alg:libcod} yields, in expectation, an $\epsilon$-stationary point of \eqref{eq:problem}, i.e.,
    \[
     \E\left[\dist(0, \partial\varphi(x_{k+1})) \right] \leq \epsilon,
    \]
    within  
    \begin{gather*}
        \left(\frac{4\left(L^f+ L^F \sqrt{2L_h} \sqrt{\varphi(x_0)- \varphi^*} + \bar{\beta}\right)^2}{p_{\min} \beta_{\min} \epsilon^2} + \frac{2 L_h M^F L^F}{p_{\min} \beta_{\min} \epsilon} \right)\left(\varphi(x_0) - \varphi^*\right)
    \end{gather*}  
    iterations.
\end{theorem}

\begin{proof}
    From \cref{propos:asymtotic:monotone}, we know that  $\E\left[\dist(0, \partial\varphi(x_{k})) \right] \to 0$ as $k\to \infty$. Let $\epsilon>0$ and let $N\geq 0$ be the first integer such that
    \begin{equation}
    \E\left[\dist(0, \partial\varphi(x_{N+1})) \right] \leq \epsilon.
    \end{equation}
    Using  \eqref{bounde_gradient1} from  \cref{bound_subgradient}, it follows that at each iteration, $k \in [0:N]$, we encounter one of the following two cases:\\
    \noindent \textit{Case 1: } If the following holds
    \[
        \left(L^f+ L^F \sqrt{2L_h} \sqrt{\varphi(x_0)- \varphi^*} + \bar{\beta}\right)\|\Delta x_{k+1}\|\geq \frac{L_h M^F L^F}{2}\|\Delta x_{k+1}\|^2,
    \]
    then we get
    \begin{align} \label{case1}
        & \dist^2\left(0, \partial_{i_k} \varphi (x_{k+1}) \right) \leq 2 \left(L^f+ L^F \sqrt{2L_h} \sqrt{\varphi(x_0)- \varphi^*} + \bar{\beta}\right)^2\|\Delta x_{k+1}\|^2 \nonumber\\
        &\overset{\eqref{decrease}}{\leq} \frac{4\left(L^f+ L^F \sqrt{2L_h} \sqrt{\varphi(x_0)- \varphi^*} + \bar{\beta}\right)^2}{\beta_{k+1}}\left(\varphi(x_{k}) - \varphi(x_{k+1})\right).
    \end{align}
    \noindent \textit{Case 2: } Otherwise,  the following is valid 
    \[
        \left(L^f+ L^F \sqrt{2L_h} \sqrt{\varphi(x_0)- \varphi^*} + \bar{\beta}\right)\|\Delta x_{k+1}\|< \frac{L_h M^F L^F}{2}\|\Delta x_{k+1}\|^2,
    \]
    which yields
    \begin{equation}\label{case2}
        \dist\left(0, \partial_{i_k} \varphi (x_{k+1}) \right) \leq L_h M^F L^F \|\Delta x_{k+1}\|^2\overset{\eqref{decrease}}{\leq} \frac{2 L_h M^F L^F}{\beta_{k+1}}\left(\varphi(x_{k}) - \varphi(x_{k+1})\right).
    \end{equation}
    Define $\mathcal{I}_1$ as the set of iterations $k \in [0:N-1]$ at which   \textit{Case 1} holds, and  $\mathcal{I}_2$ as the set of iterations $k\in [0:N-1]$ at which   \textit{Case 2} holds. Clearly: $N=|\mathcal{I}_1|+|\mathcal{I}_2|+1$. We first derive an upper bound for $|\mathcal{I}_1|$.  Let $k \in \mathcal{I}_1$, we have
    \allowdisplaybreaks
    \begin{equation}
        \label{criticality1_squared}
        \begin{split}
        &\dist^2(0, \partial\varphi(x_{k+1}))
        =
        \sum_{i=1}^{n_{\text{block}}}{\dist^2 \left(0, \partial_i\varphi(x_{k+1})\right)}
        \\
        &
        \leq
        \frac{1}{p_{\min}}\sum_{i=1}^{n_{\text{block}}}{p_i \; \dist^2 \left(0, \partial_i\varphi(x_{k+1})\right)}
        \\
        &
        = \frac1p_{\min}  \E_{k}\left[\dist^2 \left(0, \partial_{i_k}\varphi(x_{k+1})\right)\right]
        \\
        &
        \overset{\eqref{case1} }{\leq }
        \frac1p_{\min}  \frac{4\left(L^f+ L^F \sqrt{2L_h} \sqrt{\varphi(x_0)- \varphi^*} + \bar{\beta}\right)^2}{\beta_{k+1}}      \E_{k} \left[\varphi(x_{k}) - \varphi(x_{k+1}) \right]
        \\
        &
        \leq
        \frac{4\left(L^f+ L^F \sqrt{2L_h} \sqrt{\varphi(x_0)- \varphi^*} + \bar{\beta}\right)^2}{p_{\min} \beta_{\min}}      
        \left(\varphi(x_k) - \E_{k} [\varphi(x_{k+1})]\right).
        \end{split}
    \end{equation}
    Taking the expectation in the inequality \eqref{criticality1_squared} and summing it over $\mathcal{I}_1$ yields
    \begin{align*}
        |\mathcal{I}_1|\epsilon^2  & < \sum_{k\in\mathcal{I}_1} \left(\E\left[\dist(0, \partial\varphi(x_{k+1}))\right]\right)^2\leq \sum_{k\in\mathcal{I}_1} \E\left[\dist^2(0, \partial\varphi(x_{k+1}))\right]  \\
        & \leq  \sum_{k\in\mathcal{I}_1}  \frac{4\left(L^f+ L^F \sqrt{2L_h} \sqrt{\varphi(x_0)- \varphi^*} + \bar{\beta}\right)^2}{p_{\min} \beta_{\min}}      
        \E\left[\varphi(x_k) - \E_{k} [\varphi(x_{k+1})]\right] \\
        & = \frac{4\left(L^f+ L^F \sqrt{2L_h} \sqrt{\varphi(x_0)- \varphi^*} + \bar{\beta}\right)^2}{p_{\min} \beta_{\min}}   \sum_{k\in\mathcal{I}_1}      
        \left(\E\left[\varphi(x_k)\right] - \E [\varphi(x_{k+1})]\right)\\
        & \overset{\eqref{decrease}}{\leq}  \frac{4\left(L^f+ L^F \sqrt{2L_h} \sqrt{\varphi(x_0)- \varphi^*} + \bar{\beta}\right)^2}{p_{\min} \beta_{\min}}   \sum_{k=0}^{N-1}      
        \left(\E\left[\varphi(x_k)\right] - \E [\varphi(x_{k+1})]\right)  \\
        & = \frac{4\left(L^f+ L^F \sqrt{2L_h} \sqrt{\varphi(x_0)- \varphi^*} + \bar{\beta}\right)^2}{p_{\min} \beta_{\min}}  \left(\E[\varphi(x_0)] - \E[\varphi(x_N)]\right)  \\ 
        & \leq \frac{4\left(L^f+ L^F \sqrt{2L_h} \sqrt{\varphi(x_0)- \varphi^*} + \bar{\beta}\right)^2}{p_{\min} \beta_{\min}}  \left(\varphi(x_0) - \varphi^*\right) .
    \end{align*}
    Thus, we have: $|\mathcal{I}_1|< \frac{4\left(L^f+ L^F \sqrt{2L_h} \sqrt{\varphi(x_0)- \varphi^*} + \bar{\beta}\right)^2}{p_{\min} \beta_{\min} \epsilon^2}  \left(\varphi(x_0) - \varphi^*\right)$. Similarly, we derive an upper bound for $|\mathcal{I}_2|$. Let $k\in \mathcal{I}_2$, we have
    \allowdisplaybreaks
    \begin{equation}
        \label{criticality1}
        \begin{split}
        \dist(0, \partial\varphi(x_{k+1}))
        &
        \leq
        \sum_{i=1}^{n_{\text{block}}}{\dist(0, \partial_i\varphi(x_{k+1}))}
        \\
        &
        \leq
        \frac{1}{p_{\min}}\sum_{i=1}^{n_{\text{block}}}{p_i \; \dist(0, \partial\varphi_i(x_{k+1}))}
        \\
        &
        = \frac1p_{\min}  \E_{k}\left[\dist(0, \partial\varphi_{i_k}(x_{k+1}))\right]
        \\
        &
        \overset{\eqref{case2} }{\leq }
        \frac1p_{\min}  \frac{2 L_h M^F L^F}{\beta_{k+1}}      \E_{k} \left[\varphi(x_{k}) - \varphi(x_{k+1}) \right]
        \\
        &
        \leq
         \frac{2 L_h M^F L^F}{p_{\min} \beta_{\min}}      
        \left(\varphi(x_k) - \E_{k} [\varphi(x_{k+1})]\right).
        \end{split}
    \end{equation}
    Taking the expectation in the inequality \eqref{criticality1} and summing it over $\mathcal{I}_2$ yields
    \begin{align*}
        &  |\mathcal{I}_2|\epsilon  < \sum_{k\in\mathcal{I}_2} \E\left[\dist(0, \partial\varphi(x_{k+1}))\right]  \leq  \sum_{k\in\mathcal{I}_2} \frac{2 L_h M^F L^F}{p_{\min} \beta_{\min}} \E\left[\varphi(x_k) - \E_{k} [\varphi(x_{k+1})]\right]\\
        &  = \frac{2 L_h M^F L^F}{p_{\min} \beta_{\min}} \sum_{k\in\mathcal{I}_2}  \left(\E\left[\varphi(x_k)\right] - \E [\varphi(x_{k+1})]\right)\\
        & \overset{\eqref{decrease}}{\leq}  \frac{2 L_h M^F L^F}{p_{\min} \beta_{\min}} \sum_{k=0}^{N-1}  \left(\E\left[\varphi(x_k)\right] - \E [\varphi(x_{k+1})]\right)\\
        & \leq  \frac{2 L_h M^F L^F}{p_{\min} \beta_{\min}}   \left(\E\left[\varphi(x_0)\right] - \E [\varphi(x_{N+1})]\right)\leq \frac{2 L_h M^F L^F}{p_{\min} \beta_{\min}}   \left(\varphi(x_0) - \varphi^*\right).
    \end{align*}
    Therefore, we obtain: $|\mathcal{I}_2|< \frac{2 L_h M^F L^F}{p_{\min} \beta_{\min} \epsilon}   \left(\varphi(x_0) - \varphi^*\right) $.  Consequently, we obtain
    \[
        N \leq  \left(\frac{4\left(L^f+ L^F \sqrt{2L_h} \sqrt{\varphi(x_0)- \varphi^*} + \bar{\beta}\right)^2}{p_{\min} \beta_{\min} \epsilon^2}  + \frac{2 L_h M^F L^F}{p_{\min} \beta_{\min} \epsilon} \right)\left(\varphi(x_0) - \varphi^*\right).
        \]
    This proves our claim. \hfill\qed
\end{proof}
\begin{remark}
    If the sampling is deterministic, given the outer semicontinuity and closed-valuedness of subdifferentials of proper and lsc functions (see, e.g., \cite{clasonvalkonen2020nonsmooth}), \ref{propos:asymtotic:monotone} shows that any limiting point $x^*$ of $(x_k)_{k \in \N}$ is a critical point of \eqref{eq:problem} in the sense of Mordukhovich subdifferentials
\end{remark}

\begin{remark}
    The Clarke subdifferential could also be used, if the nonsmooth function $g$ were Lipschitz; the critical features of the chosen nonconvex subdifferential are compatibility with the convex subdifferential on convex functions $g$, compatibility with the Fréchet derivative on the smooth nonconvex function $\ell := f + h \circ F$, a Fermat principle, and the exact sum rule $\partial[\ell+g](x) = \nabla \ell(x)+\partial g(x)$.
    Then any solution $x^*$ of \eqref{eq:problem} must satisfy $0 \in \partial \varphi (x_{k+1})$; see, e.g., \cite{clasonvalkonen2020nonsmooth}.
\end{remark}
Note that if one chooses $\beta_{\min} = \mathcal{O}\left(\sqrt{L_h}\right)$, the iteration complexity required by \ref{alg:libcod} to achieve, in expectation, an $\epsilon$-stationary point of \eqref{eq:problem}, i.e.,  
\[
    \E\left[\dist(0, \partial\varphi(x_{k+1})) \right] \leq \epsilon,
\]  
is given by \( N = \mathcal{O}\left(\frac{\sqrt{L_h}}{\epsilon^{2}}\right) \) iterations. This matches the optimal iteration complexity established for the full-block setting in \cite{marumo2024accelerated}.

\section{Nonmonotone linearized block coordinate descent method}\label{sec4}
In the previous section, we established global convergence to a stationary point of problem \eqref{eq:problem} under the assumption that the iterates are bounded, a condition that can be challenging to verify. Inspired by the works of \cite{kanzow2024convergence,kanzow2022convergence}, this section introduces a nonmonotone variant of the proposed method, which eliminates the need for this boundedness assumption.
We present our nonmonotone LiBCoD method for solving problem \eqref{eq:problem}.

\begin{algorithm*}[ht]
\caption{Nonmonotone LiBCoD}
\label{alg:nonmonotone-libcod}
\SetAlgoNlRelativeSize{-1} 

\KwIn{$x_0$, $\beta_1 \geq \frac{\beta_{\min}}{2}>0$, $p_{\min}, u_{\min} \in (0,1)$, and $\mathcal{R}_0 = \varphi(x_0)$}
Set $k = 0$

\While{termination criterion not satisfied}{
  Pick a random index $i_k$ with probability $p_{i_k} \geq p_{\min}$

    \Repeat(
    ){
    \textcolor{blue}{$\varphi(x_{k+1}) \leq  \mathcal{R}_{k} - \frac{\beta_{k+1}}{2}\|x_{k+1} - x_k\|^{2}$}
  }
  {
    $\beta_{k+1} := 2\beta_{k+1}$
    
    $s_{k+1} := \arg\min\limits_{s \in \mathbb{R}^{n_{i_k}}} \resizebox{!}{2ex}{$\langle\nabla_{i_k} f(x_k),s-x_k^{i_k}\rangle + h\left(F(x_k)+\langle\nabla_{i_k} F(x_k),s-x_k^{i_k}\rangle\right) + g_{i_k}(s) + \frac{\beta_{k+1}}{2}\|s - x^{i_k}\|^2$}$ 

    $x^{i_k}_{k+1} := s_{k+1}$ and $x^{i}_{k+1} = x_k^i$ for $i \neq i_k$
  }\label{line:line_search_nonm}

  Choose $u_{k+1} \in (u_{\min}, 1)$ and set $\mathcal{R}_{k+1} = (1 - u_{k+1}) \mathcal{R}_k + u_{k+1} \varphi(x_{k+1})$

  Set $k = k + 1$

  $\beta_{k+1} = \max\left\{\frac{\beta_k}{4}, \frac{\beta_{\min}}{2}\right\}$
}
\end{algorithm*}

\noindent Note that the subproblem in Nonmonotone LiBCoD \ref{alg:nonmonotone-libcod} has the same structure as that in the monotone LiBCoD \ref{alg:libcod}. Consequently, solving this subproblem involves the same computational complexity as in the monotone variant. However, unlike Monotone LiBCoD \ref{alg:libcod}, a strict decrease in the objective function \(\varphi(\cdot)\) along the iterates is not required in the nonmonotone case (see Nonmonotone LiBCoD \ref{alg:nonmonotone-libcod}). This relaxation may lead to smaller regularization parameters \(\beta_{k+1}\), which, in turn, can enable larger step sizes compared to the monotone case.
 In the next lemma, we show that the inner process, in Nonmonotone LiBCoD \ref{alg:nonmonotone-libcod}, is well defined.
 
\begin{lemma}
    \label{th:kkt}
    Let \cref{ass:1} be satisfied and let $(x_k)_{k\geq 0}$ be generated by Nonmonotone LiBCoD \ref{alg:nonmonotone-libcod} and assume that $g$ is bounded from below by an affine function. Then, the following statements hold for all $k\geq 0$:
    \begin{enumerate}
        \item The inner process in Nonmonotone LiBCoD \ref{alg:nonmonotone-libcod} is well defined.
        \item $\mathcal{R}_k \geq \varphi(x_k)$.
    \end{enumerate}
\end{lemma}
The formal proof of the above lemma follows similar reasoning as in \cite{kanzow2022convergence,de2023proximal}. 
To establish the global convergence rate to a stationary point, we first provide several key results. In the following theorem, we demonstrate that the constructed sequence \((\mathcal{R}_k)_{k \geq 0}\) serves as a Lyapunov function, meaning it is a decreasing sequence.
\begin{theorem}
    \label{sec:nonmon:artf_decreas}
    Let $(x_k)_{k\geq 0}$ be generated by Nonmonotone LiBCoD \ref{alg:nonmonotone-libcod} and assume that $g$ is bounded from bellow by an affine function. Then, the following statements hold for all $k\geq 0$
    \begin{enumerate}
        \item The sequence $(\mathcal{R}_k )_{k\geq 0}$ is monotonically decreasing and satisfies:
        \begin{align*}
        \mathcal{R}_{k+1} \leq \mathcal{R}_{k} - \frac{u_{k+1}\beta_{k+1}}{2}\|x_{k+1} - x_k\|^2.
        \end{align*}
        \item For all $k\geq 0$, we have $x_k \in \mathcal{L}_{\varphi}(x_0): = \{x: \varphi(x) \leq \varphi(x_0)\}$.
    \end{enumerate}
\end{theorem}
\begin{proof}
    From Nonmonotone LiBCoD \ref{alg:nonmonotone-libcod}, we have
    \begin{align*}
      u_{k+1}\varphi(x_{k+1}) \leq  u_{k+1}\mathcal{R}_{k} - \frac{u_{k+1}\beta_{k+1}}{2}\|x_{k+1} - x_k\|^{2}.
    \end{align*}
    Combining this inequality with the definition of $\mathcal{R}_{k+1}$ yields
    \begin{align*}
        \mathcal{R}_{k+1} - (1-u_{k+1})\mathcal{R}_k \leq u_{k+1}\mathcal{R}_{k} - \frac{u_{k+1}\beta_{k+1}}{2}\|x_{k+1} - x_k\|^{2}.
    \end{align*}
    This readily proves the first assertion.
    The second one follows from
    \[
        \varphi(x_{k+1}) \leq \mathcal{R}_{k+1} \leq \mathcal{R}_k - \frac{u_{k+1}\beta_{k+1}}{2}\|x_{k+1} - x_k\|^{2}\leq \mathcal{R}_0 = \varphi(x_0)
    \]
    holding for all $k \ge 0$.  \hfill\qed   
\end{proof}
Assuming that the level set $\mathcal{L}_{\varphi}(x_0)$ is compact, then from the previous theorem, it follows that the iterates, $(x_k)_{k\geq 0}$, generated by nonmonotone LiBCoD method are bounded. Hence there exist $M^{F}_{i}$ such that $\|\nabla_{i}F(x_k)\| \leq M^{F}_{i}$ for all $k\geq 0$ and $i=1.\ldots,m$.  
The following lemma, which is analogous to \ref{bound_subgradient}, demonstrates that the subgradient of $\varphi(\cdot)$ in problem \eqref{eq:problem}, with respect to the chosen block, can be bounded by the difference between consecutive iterates.
\begin{lemma}\label{lem:4}
    Let $(x_k)_{k\geq 0}$ be generated by Nonmonotone LiBCoD \ref{alg:nonmonotone-libcod}. If \cref{ass:1} holds on $\mathcal{L}_{\varphi}(x_0)$, which is assumed to be compact, then we have
    \begin{equation}\label{bounde_gradient2}
        \begin{split}
             \dist\left(0, \partial_{i_k} \varphi (x_{k+1})  \right)
             & \leq  \left( L^f_{i_k} + L^F_{i_k} \sqrt{2L_h} \sqrt{ \varphi(x_0) - \varphi^*} + \beta_{k+1} \right)  \|\Delta x_{k+1}\|\\
             \MoveEqLeft[-1] + \frac{L_h M^{F}_{i_k}L^{F}_{i_k}}{2} \|\Delta x_{k+1}\|^2.
        \end{split}
    \end{equation}
\end{lemma}
Before deriving the iteration complexity, we first establish the asymptotic convergence in expectation of Nonmonotone LiBCoD \ref{alg:nonmonotone-libcod}. The proof proceeds along similar lines as that of \cref{propos:asymtotic:monotone}.
\begin{proposition}[Asymptotic convergence]
    \label{thm:asymtotic:N_monotone}
    Let $(x_k)_{k\geq 0}$ be generated by Nonmonotone LiBCoD \ref{alg:nonmonotone-libcod}. If \cref{ass:1} holds on $\mathcal{L}_{\varphi}(x_0)$, and $g$ is bounded from below by an affine function, then we have 
    \[
        \lim_{k\to \infty} \E\left[\dist(0, \partial\varphi(x_{k})) \right] = 0.
    \]
\end{proposition}
Next, we derive the following convergence rate. Next, we derive the following convergence rate. The proof follows a similar reasoning as that of \cref{thm:convergence:monotone}.
\begin{theorem}[Iteration complexity]
    \label{thm:convergence:nonmonotone}
    Let $(x_k)_{k\geq 0}$ be generated by Nonmonotone LiBCoD \ref{alg:nonmonotone-libcod}. If \cref{ass:1} holds on $\mathcal{L}_{\varphi}(x_0)$, and $g$ is bounded from below by an affine function, then, for any $\epsilon>0$,  Nonmonotone LiBCoD \ref{alg:nonmonotone-libcod} yields, in expectation, an $\epsilon$ first-order solution of \eqref{eq:problem}, i.e.,
    \[
     \E\left[\dist(0, \partial\varphi(x_{k+1})) \right] \leq \epsilon,
    \]
    within  
    \begin{gather*}
            \left(\frac{4\left(L^f+ L^F \sqrt{2L_h} \sqrt{\varphi(x_0)- \varphi^*} + \bar{\beta}\right)^2}{p_{\min} u_{\min} \beta_{\min} \epsilon^2} + \frac{2 L_h M^F L^F}{p_{\min} u_{\min}\beta_{\min} \epsilon} \right)\left(\varphi(x_0) - \varphi^*\right)
     \end{gather*}
    iterations.  
\end{theorem}
The complexity of Nonmonotone LiBCoD \ref{alg:nonmonotone-libcod} to achieve, in expectation, an $\epsilon$-stationary point of \eqref{eq:problem}, i.e.,  
\[
    \E\left[\dist(0, \partial\varphi(x_{k+1})) \right] \leq \epsilon,
\]  
is of the same order as that of Monotone LiBCoD \ref{alg:libcod}, with the only difference being an additional constant factor of \(\frac{1}{u_{\min}}>1\) in the nonmonotone case.
\section{Optimization Problems with Equality Constraints}
\label{sec5}
In this section, we tackle the following nonconvex, nonsmooth optimization problem with nonlinear equality constraints using a quadratic penalty method:
\begin{align}  \label{eq:constr_prob}
    &\min_{x} f(x) + \sum_{i=1}^{n_{\text{block}}} g_{i}(x^i) \nonumber \\
    &\text{s.t.} \quad F(x) = 0.
\end{align}
Problems of the form \eqref{eq:constr_prob} appear in several applications, namely machine learning \cite{tepper2018clustering,sahin2019inexact} image processing \cite{fessler2020optimization}, and many other fields. We are interested in (approximate) first-order  solutions (KKT) of \eqref{eq:constr_prob}. Hence, let us introduce the following definitions:

\begin{definition}[First-order solution and $\epsilon$ first-order solution of \eqref{eq:constr_prob}]
    \label{firstorder}
    The vector $x^*$ is a first-order solution (KKT point  ) of \eqref{eq:constr_prob} if $\exists \lambda^* \in \mathbb{R}^m$ such that
    \begin{equation*}
    \begin{split}
            -\nabla f(x^*) - \nabla F(x^*)^T \lambda^* \in \partial g(x^*)\quad \text{and} \quad F(x^*) = 0.
    \end{split}
    \end{equation*}
    Moreover, $x^*$ is an $\epsilon$-first-order solution ($\epsilon$- KKT point) of \eqref{eq:constr_prob} if $\exists \lambda^* \in \mathbb{R}^m$ such that:
    \begin{equation*}
       \begin{split}
            \dist\left(-\nabla f(x^*) + \nabla F(x^*)^T \lambda^*, \partial g(x^*)\right) \leq  \epsilon\quad
            \text{and} \quad \|F(x^*)\| \leq  \epsilon,
        \end{split}
    \end{equation*}
    where $\partial g(x^*) = \prod_{i=1}^{n_{\text{block}}} \partial g_i(x^*)$.
\end{definition}
In this section, in addition to \cref{ass:1}, we also assume the following
\begin{assumption}\label{ass:const}
    Given a compact set $\mathcal{C}$, there exists $\sigma > 0$ such that for all $x\in \mathcal{C}$
    \begin{equation} \label{contsr_qualif}
        \sigma \|F(x)\| \leq  \dist \left( - \nabla F(x)^T F(x), \prod_{i=1}^{n_{\text{block}}} \partial^{\infty} g_i(x^i) \right),
    \end{equation}
    where $\partial^{\infty} g_i$ denotes the horizon subdifferential of $g_i$ for all $i \in \{1, \ldots, n_{\text{block}}\}$.
\end{assumption}
Note that this assumption generalizes the Linear Independent Constraint Qualification (LICQ), commonly assumed in the literature, to problems with non-Lipschitz objective functions.  Such an assumption is also considered for one-block problems in \cite{Bodard2024inexact,nabou2025regularized,ElBourkhissi2025}.
The quadratic penalty function associated with  \eqref{eq:constr_prob} is
\begin{equation} \label{penalty_function}
    \varphi_{\rho}(x) := f(x) + \frac{\rho}{2}\|F(x)\|^2 + \sum_{i=1}^{n_{\text{block}}} g_{i}(x^i).
\end{equation}
Note that minimizing $\varphi_{\rho}(\cdot)$ is a special case of problem \eqref{eq:problem} with $h(\cdot) = \frac{\rho}{2}\|\cdot\|^2$, where $L_h = \rho$. Therefore, one can use Monotone LiBCoD \ref{alg:libcod} and/or Nonmonotone LiBCoD \ref{alg:nonmonotone-libcod} to minimize $\varphi_{\rho}(x)$. In the next theorem, we provide the iteration complexity of Monotone LiBCoD \ref{alg:libcod} and/or Nonmonotone LiBCoD \ref{alg:nonmonotone-libcod} to yield an $\epsilon$-KKT point of problem \eqref{eq:constr_prob}.
\begin{theorem}[Iteration Complexity]
    \label{thm:convergence:penalty}
    Let $(x_k)_{k \geq 0}$ be generated by Monotone LiBCoD \ref{alg:libcod} or Nonmonotone LiBCoD \ref{alg:nonmonotone-libcod} while minimizing $\varphi_{\rho}(x)$, as defined in \eqref{penalty_function}, and let $\beta_{\min}$ be chose such that $\beta_{\min} = \mathcal{O}\left(\sqrt{\rho}\right)$. Assume the following
    \begin{itemize}
        \item For Monotone LiBCoD \ref{alg:libcod}, \cref{ass:1} holds on a compact set to which the iterates belong.
        \item For Nonmonotone LiBCoD \ref{alg:nonmonotone-libcod}, \cref{ass:1} holds on the level set $\mathcal{L}_{\varphi_{\rho}}(x_0)$, and $g$ is bounded below by an affine function.
    \end{itemize}
    Then, for any $\epsilon > 0$, both Monotone LiBCoD \ref{alg:libcod} and Nonmonotone LiBCoD \ref{alg:nonmonotone-libcod} takes at most 
    \[
        \bar{K} = \mathcal{O}\left(\frac{\rho}{\beta_{\min} \epsilon^2}\right)
    \]  
    iterations to yield, in expectation, an $\epsilon$-stationary point, $\bar{x}_{k^*}$, of the minimization problem for $\varphi_{\rho}(\cdot)$, i.e., there exists $k^* \leq \bar{K}$ such that 
    \begin{align*}
        &\E\left[\dist \left(-\nabla f(\bar{x}_{k^*}) - \rho \nabla F(\bar{x}_{k^*})^T F(\bar{x}_{k^*}),  \partial g(\bar{x}_{k^*})\right)\right] 
        = \mathbb{E}\left[\dist(0, \partial \varphi_{\rho}(\bar{x}_{k^*}))\right] \leq \epsilon.
    \end{align*}
    Furthermore, if $\rho$ is chosen such that $\rho \geq \frac{M_f + 2}{\sigma \epsilon}$, then $\bar{x}_{k^*}$ is, in expectation, an $\epsilon$-KKT point of problem \eqref{eq:constr_prob}. In this case, the iteration complexity of both Monotone LiBCoD \ref{alg:libcod} and Nonmonotone LiBCoD \ref{alg:nonmonotone-libcod} becomes 
    \[
    \bar{K} = \mathcal{O}\left(\frac{\sqrt{\rho}}{\epsilon^2}\right) = \mathcal{O}\left(\frac{1}{\epsilon^{2.5}}\right).
    \]
\end{theorem}
\begin{proof}
    By using a monotone/nonmonotone LiBCoD algorithm to minimize $\varphi_{\rho}(\cdot)$, it follows from \cref{thm:convergence:monotone}/\cref{thm:convergence:nonmonotone} that Monotone LiBCoD \ref{alg:libcod}/Nonmonotone LiBCoD \ref{alg:nonmonotone-libcod} takes at most $\bar{K}= \mathcal{O}\left(\frac{L_h}{\beta_{\min}\epsilon^2}\right) = \mathcal{O}\left(\frac{\sqrt{\rho}}{\epsilon^2}\right)$ (as $L_h=\rho$ and $\beta_{\min} = \mathcal{O}(\sqrt{L_h})$) iterations to yield, in expectation, an $\epsilon$-stationary point, $\bar{x}_{k^*}$, of the minimization problem for $\varphi_{\rho}(\cdot)$, i.e., there exists $k^* \leq \bar{K}$ such that
    \begin{align} \label{stat_penal}
        &\E\left[\dist \left(-\nabla f(\bar{x}_{k^*}) - \nabla F(\bar{x}_{k^*})^T \nabla h(F(\bar{x}_{k^*})), \partial g(\bar{x}_{k^*})\right)\right] \nonumber \\
        &= \E\left[\dist \left(-\nabla f(\bar{x}_{k^*}) - \rho \nabla F(\bar{x}_{k^*})^T F(\bar{x}_{k^*}), \prod_{i=1}^{n_{\text{block}}} \partial g_i(\bar{x}^i_{k^*})\right)\right] \leq \epsilon.
    \end{align}
    On the other hand, we have
    \begin{align*}
        &\rho\, \E\left[\dist \left( - \nabla F(\bar{x}_{k^*})^T F(\bar{x}_{k^*}), \prod_{i=1}^{n_{\text{block}}} \partial^{\infty} g_i(\bar{x}^i_{k^*}) \right)\right] \\
        &= \E\left[\dist \left( - \rho \nabla F(\bar{x}_{k^*})^T F(\bar{x}_{k^*}), \prod_{i=1}^{n_{\text{block}}} \partial^{\infty} g_i(\bar{x}^i_{k^*}) \right) \right]\\
        & \leq \E\left[ \dist \left(-\nabla f(\bar{x}_{k^*}) - \rho \nabla F(\bar{x}_{k^*})^T F(\bar{x}_{k^*}), \prod_{i=1}^{n_{\text{block}}} \partial g_i(\bar{x}^i_{k^*})\right)\right]\\
        \MoveEqLeft[-1]+ \E\left[\|\nabla f(\bar{x}_{k^*})\|\right] + 1 \\
        & \overset{\eqref{stat_penal}}{\leq} \epsilon + M^f + 1,
    \end{align*}
    where $M^f > 0$ is such that $\E\left[\|\nabla f(x)\|\right]\leq M^f$ for any $x$ in some compact set. Furthermore, using \eqref{contsr_qualif} and choosing $\rho \geq \frac{M_f + 2}{\sigma\epsilon}$, it follows 
    \begin{equation}\label{feasib}
      \E\left[\|F(\bar{x}_{k^*})\|\right] \leq \frac{\epsilon + M_f + 1}{\rho\sigma} \leq \epsilon.
    \end{equation}
    Hence, from \eqref{stat_penal} and \eqref{feasib}, it follows that there exists $\bar{\lambda}_{k^*} = \rho F(\bar{x}_{k^*})$ such that when $\rho \geq \frac{M_f + 2}{\sigma\epsilon}$, $\bar{x}_{k^*}$ is, in expectation, an $\epsilon$-KKT point of problem \eqref{eq:constr_prob}, obtained via the monotone/nonmonotone LiBCoD algorithm after $\mathcal{O}\left(\frac{\sqrt{\rho}}{\epsilon^2}\right) = \mathcal{O}\left(\frac{1}{\epsilon^{2.5}}\right)$ Jacobian evaluations.\hfill\qed
\end{proof}
 Note that this iteration complexity matches the optimal one derived for quadratic penalty methods in the full-block setting in \cite{Bourkhissi2024ComplexityOL}.
\section{Numerical simulations}\label{sec6}
In this section, we perform several numerical experiments to evaluate the performance of the proposed methods and compare them with existing approaches from the literature. The implementation was carried out on a MacBook M2 with 16GB of RAM, using the Julia programming language.

\subsection{Binary classification}
We focus on solving the following optimization problem  
\begin{equation}
\label{sec:sim:num_prob}
    \min_{x \in \mathbb{R}^n} \frac{1}{2} \|F(x)\|^2 + \lambda\|x\|_1,
\end{equation}
where \( F(x) := (F_1(x),\ldots,F_m(x)) \), and the \(\ell_1\)-norm regularization promotes sparsity in the solution. Clearly, problem  \eqref{sec:sim:num_prob} is of the form of \eqref{eq:problem}. 
To define \( F(x) \), we consider the following two possible nonconvex formulations \cite{zhao2010convex,tran2020stochastic}. 
The first is based on a nonconvex squared-logarithmic loss  
\begin{equation} \label{nonlinear1}
      \hspace{-0.2cm} F_i(x) = \log\left(1 + \big(y_i(a_i^T x + b_i) - 1\big)^2\right), \, i=1:m,
\end{equation}
while the second uses the standard logistic function   
\begin{equation} \label{nonlinear2}
    F_i(x) = 1 - \frac{1}{1 + \exp\big(-y_i(a_i^T x + b_i)\big)}, \,  i=1:m,
\end{equation}
where \( y_i \) represents the binary labels, \( a_i \) denotes the feature vectors, and \( b_i \) is an offset parameter.
\subsubsection{Monotone LiBCod Vs. ProxCD}
We evaluate the performance of Monotone LiBCoD \ref{alg:libcod} and compare it against the proximal block coordinate descent method (ProxCD) \cite{deng2020efficiency}, as well as the full Gauss-Newton method \cite{marumo2024accelerated}, which corresponds to using a single full block in our approach.
While ProxCD serves as a natural baseline, it does not explicitly exploit the problem’s structure—particularly the smooth coupling between blocks—which our method is specifically designed to leverage. This comparison highlights the advantage of incorporating structural information into the optimization process.
We first test our method using the Gisette dataset from the LIBSVM collection \cite{chang2011libsvm}, by solving problem  \eqref{sec:sim:num_prob} with the formulations of $F(x)$ given in \eqref{nonlinear1}, while varying the block sizes and the regularization parameter $\lambda$.
We then compare the results with those obtained using ProxCD. Both methods are terminated either after $60$ seconds or upon reaching a predefined accuracy level ($85\%$, $90\%$, or $95\%$). For each combination of the regularization parameter  \( \lambda \) and block size, we run both methods three times and report the mean and standard deviation of the CPU time and the number of epochs (expressed as a percentage of full gradient evaluations) required to achieve the target accuracy. The results are summarized in \cref{block-size-comparison,block-size-comparison_epoch}. If a method fails to reach the target accuracy within $60$ second, the corresponding entry is marked as "-". From \cref{block-size-comparison,block-size-comparison_epoch}, it can be seen that, in general, Monotone LiBCoD consistently outperforms ProxCD in terms of CPU time and epochs, demonstrating significantly faster convergence to reasonable accuracy levels. Moreover, our block-based approach, which, for instance, uses only $10\%$ of the blocks per iteration, also outperforms the full Gauss-Newton method.

In summary, our results align with theoretical expectations regarding the number of epochs: smaller block sizes generally require fewer full gradient evaluations in relative terms. However, when considering CPU time, the optimal block size is \(500\), representing \(10\%\) of the data. Nonetheless, the optimal choice may vary depending on the specific problem formulation and targeted accuracy, highlighting the impact of problem structure and computational overhead on performance.

Overall, LiBCoD demonstrates more robustness with respect to the regularization parameter \( \lambda \) compared to ProxCD. Its superior computational efficiency and stability establish LiBCoD as the preferred choice for scenarios requiring both high precision and computational efficiency. 

\begin{table}[t]
\caption{Comparison of LiBCoD and ProxCD, on formulation \eqref{nonlinear1}, for various block sizes and $\lambda$ values in CPU time (s) on the Gisette dataset.  Monotone LiBCoD consistently outperforms ProxCD in CPU time, achieving faster convergence to reasonable accuracy levels. Additionally, our block-based approach using as little as $10\%$ of the blocks per iteration outperforms the full Gauss-Newton method.}
\label{block-size-comparison}
\vskip 0.15in
\begin{center}
\begin{small}
\begin{sc}
\setlength{\tabcolsep}{5.8pt} 
\resizebox{\linewidth}{!}{
\begin{tabular}{lcccccccccccccccccccccccccc}
\toprule
$\lambda$ & Bl-size & \multicolumn{6}{c} {LiBCoD} & \multicolumn{6}{c}{ProxCD} \\
\cmidrule(lr){3-8} \cmidrule(lr){9-14}
 &  & \multicolumn{2}{c}{$85\%$} & \multicolumn{2}{c}{$90\%$} & \multicolumn{2}{c}{$95\%$} & \multicolumn{2}{c}{$85\%$} & \multicolumn{2}{c}{$90\%$} & \multicolumn{2}{c}{$95\%$} \\
 \cmidrule(lr){3-4} \cmidrule(lr){5-6}  \cmidrule(lr){7-8} \cmidrule(lr){9-10} \cmidrule(lr){11-12} \cmidrule(lr){13-14}  
 & & mean  & std & mean & std & mean  & std & mean & std &mean & std & mean & std\\
\midrule
\multirow{4}{*}{$10^{2}$}         &   5  & 13.55 &  1.26 & 53.0 & 7.5 & -  & - & - & - &  - & - &  - & - &\\
                                  &  50  & 2.60  & 1.14 & 8.50  & 3.32 & - & - & - & - & -  &  - & - & - &    \\
                                  & 500   &  \textbf{0.56} & 0.02 & \textbf{1.27} & 0.33 & -  & - & 4.68 & 0.25 &  - & - & - & - &     \\
                                  & 5000  &  4.79 & 0.51 & 4.79 & 0.51 & - & - & 9.38 & 0.86 & 13.59  &  1.24  &  - &  - &   \\
\midrule
\multirow{4}{*}{$5$}              &   5  & 13.49  & 5.38 &  31.43 & 6.34 &  - & - & 38.95 &  13.72 &  - & - & -  & - &      \\
                                  &  50  &  2.13 & 0.73 & 3.59 & 0.20 & 18.26  &  3.08 & 11.98 & 2.57 & 26.57  &  5.01 & - & - &    \\
                                  & 500   & \textbf{0.65}  & 0.07 & \textbf{0.65} & 0.07 &  \textbf{1.25} & 0.12 & 5.11 & 0.85 & 8.20  &  1.24 & 29.25 & 4.33&     \\
                                  & 5000  &  5.98 & 0.17 & 5.98 & 0.17 & 5.98  & 0.17  & 9.40 & 0.22 & 10.81  &  0.18  & 28.69 & 0.48 &   \\
\midrule
\multirow{4}{*}{$0.1$}            &   5  & 14.64  & 7.58  & 35.71 &  17.84 & -  &  - &  51.07 & 7.22 & -  & - & - & - &      \\
                                  &  50  & 1.45  & 0.49 & 3.31 & 0.92 &  16.33 & 6.38 & 9.66 & 2.38 & 25.07  &  10.42 & - & - &    \\
                                  & 500   &  \textbf{0.68} & 0.03 & \textbf{0.68} & 0.03 & \textbf{1.12}  & 0.38 & 5.10 & 0.58 &  8.53 &  0.32 & 30.18 & 6.69 &     \\
                                  & 5000  &  6.00 & 0.15 & 6.00 & 0.15 & 6.00  & 0.15 & 8.43 & 0.15 & 9.67  &   0.18 & 25.60 & 0.51 &   \\
\midrule
\multirow{4}{*}{$10^{-3}$}         &   5  & 10.72  & 3.50 & 24.58 & 3.68 &  - & - & 36.12 & 7.89 & -  & - & -  & - &      \\
                                  &  50  & 1.72  & 0.04 & 2.87 & 0.69 & 17.96  & 12.38 & 8.83 & 1.49 & 21.29  & 5.67  & - & - &    \\
                                  & 500   & \textbf{0.73}  & 0.10 &\textbf{0.73}  & 0.10 &  \textbf{1.95} & 0.45 & 4.90 & 0.39 & 9.10  & 0.57  & 31.64 & 8.85 &     \\
                                  & 5000  &  5.64 & 0.04 & 5.64 & 0.04 & 5.64  & 0.04 & 8.02 & 0.06 & 9.18  &  0.07  & 28.11 & 1.02 &   \\
\midrule
\multirow{4}{*}{$10^{-5}$}         &   5  &  9.94 & 5.61 & 23.01 & 8.05 & -  & - &  36.15 & 17.33 & -  & -  & - & - &    \\
                                  &  50  & 1.48  & 0.20 & 3.26 & 0.81 & 13.80  & 2.64  & 11.20 & 2.72  & 25.28  & 4.52 &  - & - &  \\ 
                                  & 500   &  \textbf{0.57} & 0.03 & \textbf{0.57} & 0.03 & \textbf{1.31}  & 0.86 & 4.66 & 0.48 & 8.18  & 1.70  & 31.59 & 0.81 &     \\
                                  & 5000  & 5.67  & 0.20 & 5.67 & 0.20  & 5.67  & 0.20  & 8.22 & 0.18  & 9.42  & 0.22  & 28.48 & 0.67 &   \\
\midrule                                  
\end{tabular}
}
\end{sc}
\end{small}
\end{center}
\vskip -0.1in
\end{table}
\begin{table*}[t]
\caption{Comparison of LiBCoD and ProxCD, on formulation \eqref{nonlinear1}, for various block sizes and $\lambda$ values in epoch (i.e., percentage of gradient evaluation) on the Gisette dataset. Monotone LiBCoD consistently outperforms ProxCD in epochs, achieving faster convergence to reasonable accuracy levels. Additionally, our block-based approach using as little as $10\%$ of the blocks per iteration outperforms the full Gauss-Newton method.}
\label{block-size-comparison_epoch}
\vskip 0.15in
\begin{center}
\begin{small}
\begin{sc}
\setlength{\tabcolsep}{1pt} 
\resizebox{\linewidth}{!}{
\begin{tabular}{lcccccccccccccccccccccccccc}
\toprule
$\lambda$ & Bl-size & \multicolumn{6}{c} {LiBCoD} & \multicolumn{6}{c}{ProxCD} \\
\cmidrule(lr){3-8} \cmidrule(lr){9-14}
 &  & \multicolumn{2}{c}{$85\%$} & \multicolumn{2}{c}{$90\%$} & \multicolumn{2}{c}{$95\%$} & \multicolumn{2}{c}{$85\%$} & \multicolumn{2}{c}{$90\%$} & \multicolumn{2}{c}{$95\%$} \\
 \cmidrule(lr){3-4} \cmidrule(lr){5-6}  \cmidrule(lr){7-8} \cmidrule(lr){9-10} \cmidrule(lr){11-12} \cmidrule(lr){13-14}  
 & & mean(\%)  & std(\%)  & mean(\%) & std(\%) & mean(\%)  & std(\%) & mean(\%) & std(\%) &mean(\%) & std(\%) & mean(\%) & std(\%)\\
\midrule
\multirow{4}{*}{$10^{2}$}         &   5  & \textbf{3.1} &  0.25 & \textbf{14.0} & 0.41 & -  & - & - & - &  - & - &  - & - &\\
                                  &  50  & 6.0  & 2.64 & 19.66  & 7.57 & - & - & - & - & -  &  - & - & - &    \\
                                  & 500   & 10.0  & 0.0 & 23.33 & 5.77 & -  & - & 33.33 & 5.77 &  - & - & - & - & \\
                                  & 5000  &  100.0 & 0.0 & 100.0 & 0.0 & - & - & 200.0 & 0.0 & 500.0  &  0.0  &  - &  - &   \\
\midrule
\multirow{4}{*}{$5$}              &   5  & \textbf{3.16}  & 1.25  & \textbf{7.43} & 1.55 & -  & - &  5.33 & 1.98 & -  & - & - & - &      \\
                                  &  50  & 5.0  & 1.73 & 8.33 & 0.57 &  42.66 & 7.63 & 13.66 & 3.05 & 32.0  &  6.92 & - & - &    \\
                                  & 500  & 10.0 & 0.0 & 10.0 & 0.0 & \textbf{20.0}  & 0.0 & 26.66 & 11.54 &  63.33 &  11.54 & 306.66 & 55.07 &     \\
                                  & 5000 & 100.0 & 0.0 & 100.0 & 0.0 & 100.0  & 0.0 & 200.0 & 0.0 & 300.0  &   0.0 & 1700.0 & 0.0 &   \\
\midrule
\multirow{4}{*}{$0.1$}            &   5  & 3.43  & 1.75  & 8.43 & 4.20 & -  & - &  6.86 & 0.92 & -  & - & - & - &      \\
                                  &  50  & \textbf{3.33}  & 1.15 & \textbf{7.66} & 2.08 &  37.66 & 14.18 & 10.33 & 2.51 & 29.66  &  12.66 & - & - &    \\
                                  & 500   &  10.0 & 0.0 & 10.0 & 0.0 & \textbf{16.66}  & 5.77 & 26.66 & 5.77 &  70.0 &  0.0 & 326.66 & 75.05 &     \\
                                  & 5000  &  100.0 & 0.0 & 100.0 & 0.0 & 100.0  & 0.0 & 200.0 & 0.0 & 300.0  &  0.0 & 1700.0 & 0.0 &   \\
\midrule
\multirow{4}{*}{$10^{-3}$}        &   5  & \textbf{2.73}  & 0.92 & \textbf{6.26} & 0.97 & - & - & 5.13 & 1.20 & -  & - & -  & - &      \\
                                  &  50  & 4.0  & 0.0 & 6.66 & 1.52 & 40.66  & 27.75 & 10.0 & 2.0 & 25.0  & 6.55  & - & - &    \\
                                  & 500   & 10.0  & 0.0 & 10.0  & 0.0 & \textbf{30.0} & 10.0 & 23.33 & 5.77 & 73.33  & 5.77  & 334.33 & 94.51 &     \\
                                  & 5000  &  100.0 & 0.0 & 100.0 & 0.0 & 100.0 & 0.0 & 200.0 & 0.0 & 300.0  &  0.0  & 1900.0 & 0.0 &   \\
\midrule
\multirow{4}{*}{$10^{-5}$}        &   5   &  \textbf{2.66} & 1.51 & \textbf{6.16} & 2.17 & -  & - & 5.46 & 2.61 & -  & -  & - & - &    \\
                                  &  50   & 3.66  & 5.77 & 8.0 & 1.73 & 34.33  & 8.38  & 13.33 & 2.51  & 31.66  & 4.16 &  - & - &  \\ 
                                  & 500   &  10.0 & 0.0 & 10.0 & 0.0 & \textbf{23.33}  & 15.27 & 26.66 & 5.77 & 73.33  & 20.81  & 373.33 & 5.77 & \\
                                  & 5000  & 100.0  & 0.0 & 100.0 & 0.0  & 100.0 & 0.0  & 200.0 & 0.0  & 300.0  & 0.0  & 1900.0 & 0.0 &   \\
\midrule                                  
\end{tabular}
}
\end{sc}
\end{small}
\end{center}
\vskip -0.1in
\end{table*}
\begin{table*}[t]
\caption{Comparison of LiBCoD and ProxCD, on formulation \eqref{nonlinear2}, for various block sizes and $\lambda$ values in CPU time (s) on the Gisette dataset.}
\label{block-size-comparison2}
\vskip 0.15in
\begin{center}
\begin{small}
\begin{sc}
\setlength{\tabcolsep}{5.8pt} 
\resizebox{\linewidth}{!}{
\begin{tabular}{lcccccccccccccccccccccccccc}
\toprule
$\lambda$ & Bl-size & \multicolumn{6}{c} {LiBCoD} & \multicolumn{6}{c}{ProxCD} \\
\cmidrule(lr){3-8} \cmidrule(lr){9-14}
 &  & \multicolumn{2}{c}{$85\%$} & \multicolumn{2}{c}{$90\%$} & \multicolumn{2}{c}{$95\%$} & \multicolumn{2}{c}{$85\%$} & \multicolumn{2}{c}{$90\%$} & \multicolumn{2}{c}{$95\%$} \\
 \cmidrule(lr){3-4} \cmidrule(lr){5-6}  \cmidrule(lr){7-8} \cmidrule(lr){9-10} \cmidrule(lr){11-12} \cmidrule(lr){13-14}  
 & & mean  & std & mean & std & mean  & std & mean & std &mean  & std & mean & std\\
\midrule
\multirow{4}{*}{$10^{2}$}         &   5  & -  & - & - & - &  - & - & - & - & -  & - & -  & - &      \\
                                  &  50  & 28.09 & 15.64  & - & - &  - & - & - & - & -  &  - & - & - &    \\
                                  & 500   &  10.56 & 16.58 & - & - & -  & - & - & - &  - &  - & - & - &     \\
                                  & 5000  & \textbf{3.64}  & 0.07  & - & - &  - & - & - & - & -  &  - & - & - &   \\
\midrule
\multirow{4}{*}{$5$}              &   5  & 16.82  & 9.01 & - & - &  - & - & 22.07 & 5.36 & -  & - & -  & - &      \\
                                  &  50  &  1.55 & 0.24 & 4.94 & 0.82 & -  & - & 3.71 & 0.76  & 14.03  & 2.60  & - & - &    \\
                                  & 500   & \textbf{0.60}  & 0.03 & \textbf{1.03} & 0.35 &  6.96 & 3.83 & 2.45  & 0.44 &  7.80 & 1.20  & - & - &     \\
                                  & 5000  & 4.56  & 0.26 & 4.56 & 0.26 &  \textbf{4.56} & 0.26 & 4.80 & 0.18 & 8.40  & 0.33   & -  & - &   \\
\midrule
\multirow{4}{*}{$0.1$}         &   5  &  22.30 & 9.83 & - & - & - & - & 33.42 & 17.43 & -  & - & - & - &      \\
                                  &  50  & 2.44  & 1.32 & 5.21 & 1.50 & 19.54  & 3.20 & 4.04 & 2.26 &  7.19 & 3.01  & 44.18 & 12.74 &    \\
                                  & 500   & \textbf{0.71}  & 0.04 & \textbf{0.71} & 0.04  &  \textbf{1.89} & 0.46 & 4.42 & 1.56 & 7.71  & 1.53  & 32.16  &  12.68 &     \\
                                  & 5000  & 5.81  & 0.14 & 5.81  & 0.14 &  5.81  & 0.14 & 1.57 & 0.06 & 2.98 & 0.13 & 23.73  & 0.79 &   \\
\midrule
\multirow{4}{*}{$10^{-3}$}         &   5  & 22.9  & 16.44 & - & - & -  & - & 26.47 & 19.14 &  - & - & -  & - &      \\
                                  &  50  &  1.94 &  0.64 & 4.93 & 0.82 & 22.18  & 5.29 & 3.56 & 0.85 & 8.32  & 1.99  & 40.87 & 6.90 &    \\
                                  & 500   & \textbf{0.57}  & 0.04 &  \textbf{0.57}  & 0.04  & \textbf{2.60}  & 0.67 & 1.98 & 0.54 & 4.88 & 0.85  & 30.94 & 9.95  &     \\
                                  & 5000  & 4.88  & 0.13  &  4.88 & 0.13  &  4.88 & 0.13  & 1.21 & 0.05 & 2.34  & 0.05 & 20.50 & 0.44 &   \\
\midrule
\multirow{4}{*}{$10^{-5}$}         &   5  & 19.35  & 7.39 & 48.46 & 8.13 &  - & - & 21.67 & 6.44 & -  & - & -  & - &      \\
                                  &  50  & 1.53  & 0.29 & 4.14 & 0.78 & 18.57 & 4.19 & 3.23 & 1.02 & 8.80 & 3.22 & 35.58 & 9.34 &    \\
                                  & 500   &  \textbf{0.54} & 0.02 & \textbf{0.54} & 0.02 & \textbf{2.02}  & 0.24 & 3.08 & 1.19 &  6.79 & 2.23 & 27.25 & 5.51  &     \\
                                  & 5000  & 4.76  & 0.25  & 4.76 & 0.25  &  4.76 & 0.25  & 1.20 & 0.06 & 2.40 & 0.15 & 21.76 & 0.65 &   \\
\midrule                                  
\end{tabular}
}
\end{sc}
\end{small}
\end{center}
\vskip -0.1in
\end{table*}

\begin{table*}[t]
\caption{Comparison of LiBCoD and ProxCD, on formulation \eqref{nonlinear2}, for various block sizes and $\lambda$ values in epoch (i.e., percentage of gradient evaluation) on the Gisette dataset.}
\label{block-size-comparison2_epoch}
\vskip 0.15in
\begin{center}
\begin{small}
\begin{sc}
\setlength{\tabcolsep}{1pt} 
\resizebox{\linewidth}{!}{
\begin{tabular}{lcccccccccccccccccccccccccc}
\toprule
$\lambda$ & Bl-size & \multicolumn{6}{c} {LiBCoD} & \multicolumn{6}{c}{ProxCD} \\
\cmidrule(lr){3-8} \cmidrule(lr){9-14}
 &  & \multicolumn{2}{c}{$85\%$} & \multicolumn{2}{c}{$90\%$} & \multicolumn{2}{c}{$95\%$} & \multicolumn{2}{c}{$85\%$} & \multicolumn{2}{c}{$90\%$} & \multicolumn{2}{c}{$95\%$} \\
 \cmidrule(lr){3-4} \cmidrule(lr){5-6}  \cmidrule(lr){7-8} \cmidrule(lr){9-10} \cmidrule(lr){11-12} \cmidrule(lr){13-14}  
 & & mean(\%)  & std(\%) & mean(\%) & std(\%) & mean(\%)  & iter(\%) & mean(\%) & std(\%) &mean(\%)  & iter(\%) & mean(\%) & std(\%)\\
\midrule
\multirow{4}{*}{$10^{2}$}         &   5  & -  & - & - & - &  - & - & - & - & -  & - & -  & - &      \\
                                  &  50  & \textbf{59.66} & 33.54  & - & - &  - & - & - & - & -  &  - & - & - &    \\
                                  & 500   &  193.33 & 300.38 & - & - & -  & - & - & - &  - &  - & - & - &     \\
                                  & 5000  & 100.0 & 0.0  & - & - &  - & - & - & - & -  &  - & - & - &   \\
\midrule
\multirow{4}{*}{$5$}              &   5  & 3.63  & 1.96 & - & - & - & - & 4.86 & 1.19 & -  & - & -  & - &      \\
                                  &  50  &  \textbf{3.33} & 0.57 & \textbf{10.33} & 1.52 & -  & - & 7.66 & 1.15  & 29.33  & 5.50  & - & - &    \\
                                  & 500   & 10.0  & 0.0 & 16.66 & 5.77 &  123.33 & 66.58 & 23.33  & 5.77 &  110.0 & 43.58  & - & - &     \\
                                  & 5000  & 100.0  & 0.0 & 100.0 & 0.0 &  100.0 & 0.0 & 200.0 & 0.0 & 500.0  & 0.0   & -  & - &   \\
\midrule
\multirow{4}{*}{$0.1$}            &   5  &  4.50 & 2.06 & - & - & - & - & 6.86 & 4.0 & -  & - & - & - &      \\
                                  &  50  & \textbf{4.33}  & 2.08 & \textbf{9.33} & 2.30 & 35.66  & 6.80 & 6.66 & 3.21 &  12.66 & 4.61  & 74.66 & 21.59 &    \\
                                  & 500   & 10.0  & 0.0 & 10.0 & 0.0  &  \textbf{26.66} & 5.77 & 43.33 & 15.27 & 76.66  & 11.54  & 326.66  &  98.14 &     \\
                                  & 5000  & 100.0   & 0.0 & 100.0   & 0.0 &  100.0   & 0.0 & 200.0 & 0.0 & 400.0 & 0.0 & 3200.0  & 0.0 &   \\
\midrule
\multirow{4}{*}{$10^{-3}$}        &   5  & \textbf{4.03}  & 1.95 & - & - & -  & - & 4.60 & 2.32 &  - & - & -  & - &      \\
                                  &  50  &  4.33 &  1.15 & 11.0 & 1.0 & 50.33  & 10.78 & 7.33 & 2.51 & 18.0  & 3.46  & 94.33 & 16.62 &    \\
                                  & 500  & 10.0 & 0.0 &  \textbf{10.0} & 0.0  & \textbf{46.66}  & 15.27 & 23.33 & 5.77 & 63.33 & 11.54  & 386.66 & 123.42 &     \\
                                  & 5000  & 100.0  & 0.0  &  100.0 & 0.0  &  100.0 & 0.0  & 200.0 & 0.0 & 400.0  & 0.0 & 3400.0 & 0.0 &   \\
\midrule
\multirow{4}{*}{$10^{-5}$}        &   5  & 4.40  & 1.73 & 10.83 & 1.71 & - & - & 5.0 & 1.55 & -  & - & -  & - &      \\
                                  &  50  & \textbf{3.66}  & 0.57 & \textbf{9.66} & 1.52 & 43.33 & 10.11 & 7.33 & 2.08 & 20.33 & 6.11 & 84.0 & 21.28 &    \\
                                  & 500   &  10.0 & 0.0 & 10.0 & 0.0 & \textbf{36.66}  & 5.77 & 36.66 & 15.27 &  80.0 & 62.45 & 333.33 & 58.59  &     \\
                                  & 5000  & 100.0   & 0.0  & 100.0  & 0.0  &  100.0  & 0.0  & 200.0 & 0.0 & 400.0 & 0.0 & 3400.0 & 0.0 &   \\
\midrule                                  
\end{tabular}
}
\end{sc}
\end{small}
\end{center}
\vskip -0.1in
\end{table*}

Our second set of experiments consists in running Monotone LiBCoD and comparing its performance against ProxCD in terms of the evolution of the loss function over time and iterations, using several datasets from LIBSVM collection \cite{chang2011libsvm} and different blocksizes ranges from $\{0.1\%, 1\%, 10\%, 100\%\}$. In this setup, we first run both methods for $100$ iterations (see \ref{fig:cvge1}) and then separately for a fixed predefined CPU time (see \ref{fig:cvge1cpu}). As shown in \cref{fig:cvge1,fig:cvge1cpu}, one can observe that Monotone LiBCoD consistently outperforms ProxCD, achieving lower loss values in significantly fewer iterations and reduced computational time. Moreover, one can observe that using only $10\%$ of the original block size yields performance comparable to the full Gauss-Newton method.
\begin{figure}[t]
    \centering
    \begin{minipage}{0.85\textwidth}
        \centering
        \subfigure[Gisette]{%
            \includegraphics[width=0.23\textwidth]{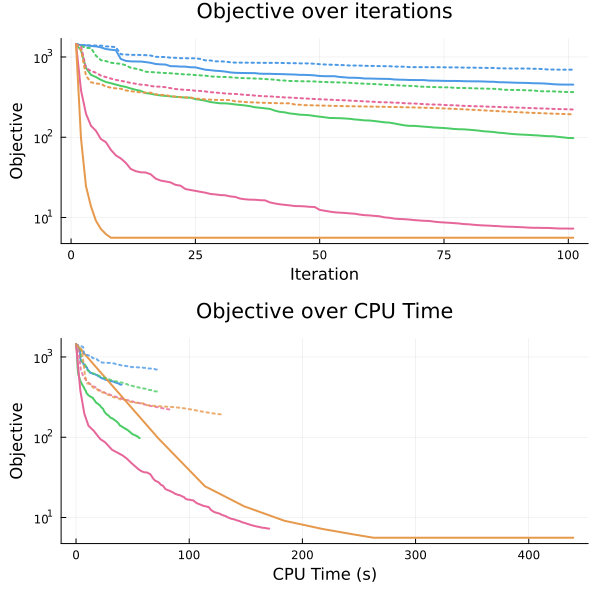}
            \label{fig:colon0}
        }\hfill
        \subfigure[Colon-Cancer]{%
            \includegraphics[width=0.23\textwidth]{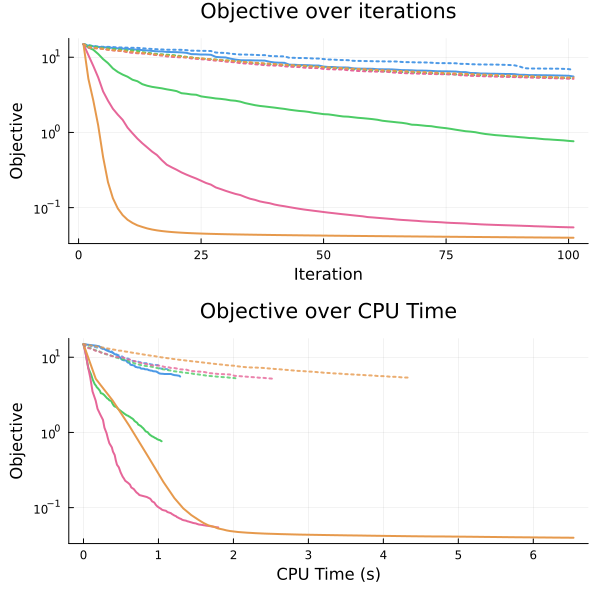}
            \label{fig:gisette0}
        }\hfill
        \subfigure[Duke]{%
            \includegraphics[width=0.23\textwidth]{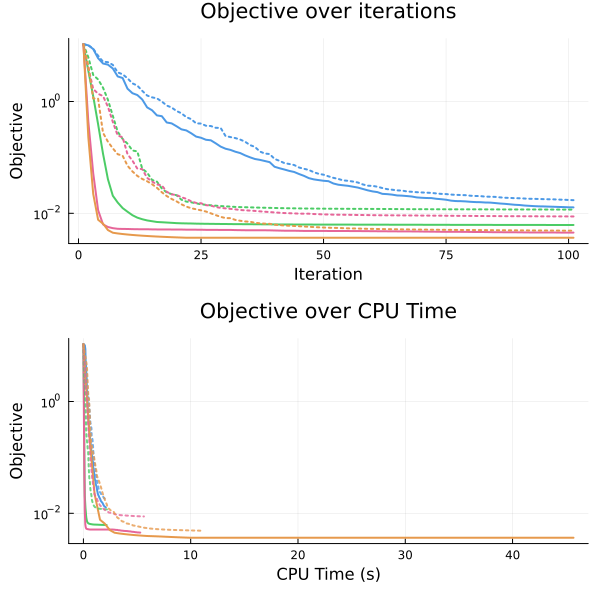}
            \label{fig:duke0}
        }\hfill
        \subfigure[Leukemia]{%
            \includegraphics[width=0.23\textwidth]{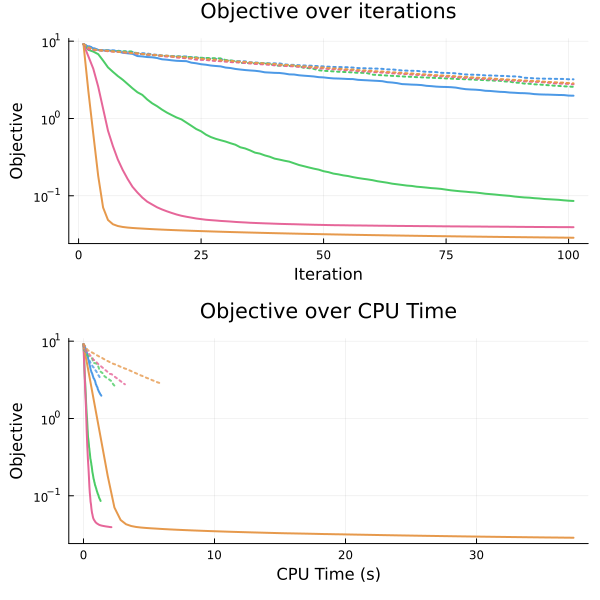}
            \label{fig:leuk0}
        }
    \end{minipage}%
    \begin{minipage}{0.15\textwidth}
        \centering
        \includegraphics[width=\textwidth]{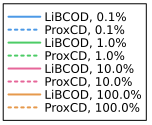}
    \end{minipage}
    \caption{Comparison between Monotone LiBCoD \ref{alg:libcod} and ProxCD with $\lambda = 0.001$ on formulation \eqref{nonlinear1}. Monotone LiBCoD consistently achieves lower loss values in fewer iterations and less CPU time. Notably, using only $10\%$ of the block size yields performance comparable to the full Gauss-Newton method.}
    \label{fig:cvge1}
\end{figure}
\begin{figure}[t]
    \centering
    \begin{minipage}{0.85\textwidth}
        \centering
        \subfigure[Gisette]{%
            \includegraphics[width=0.23\textwidth]{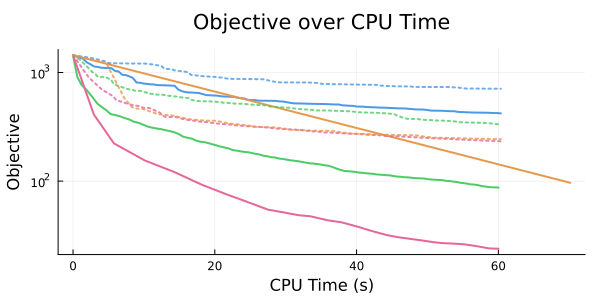}
            \label{fig:colon0cpu}
        }\hfill
        \subfigure[Colon-Cancer]{%
            \includegraphics[width=0.23\textwidth]{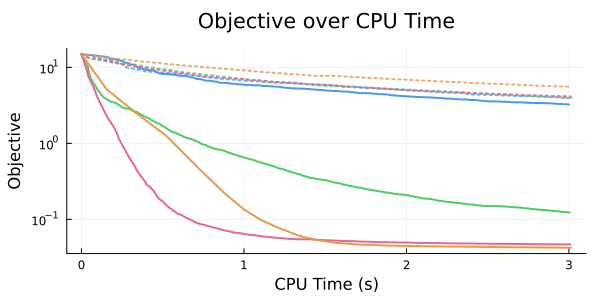}
            \label{fig:gisette0cpu}
        }\hfill
        \subfigure[Duke]{%
            \includegraphics[width=0.23\textwidth]{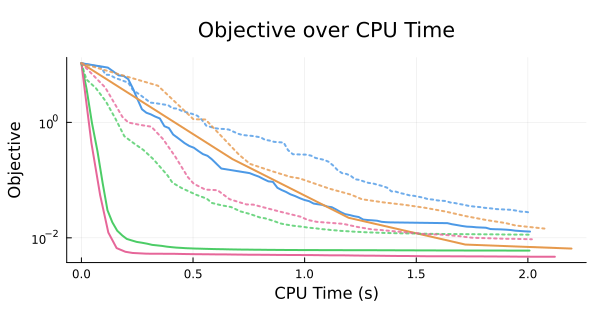}
            \label{fig:duke0cpu}
        }\hfill
        \subfigure[Leukemia]{%
            \includegraphics[width=0.23\textwidth]{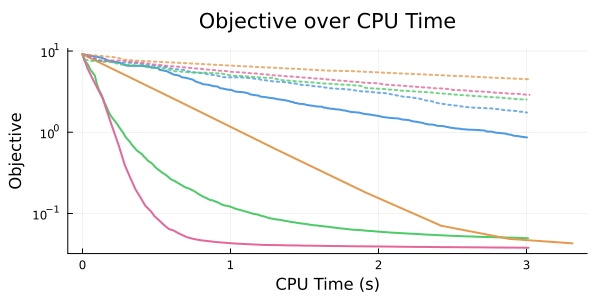}
            \label{fig:leuk0cpu}
        }
    \end{minipage}%
    \begin{minipage}{0.15\textwidth}
        \centering
        \includegraphics[width=\textwidth]{plots/monotoneVsnonmonotone/LOSS.plots/legends.png}
    \end{minipage}
    \caption{Comparison between Monotone LiBCoD \ref{alg:libcod} and ProxCD with $\lambda = 0.001$  under predefined CPU time on formulation \eqref{nonlinear1}. Monotone
    LiBCoD outperform ProxCD and consistently achieves lower loss values in less CPU time.}
    \label{fig:cvge1cpu}
\end{figure}

\subsubsection{Nonmototone LiBCoD}
Our third set of experiments focuses on evaluating the performance of the nonmonotone LiBCoD method using several datasets from the LIBSVM collection. Specifically, we run Nonmotonote LiBCoD \ref{alg:nonmonotone-libcod} with different fixed values of $u_k = u \in \{0.1, 0.5, 1\}, \forall k\geq 0$, using a regularization parameter $\lambda = 0.001$ and selecting only $10\%$ of the data at each iteration.
We then report the evolution of the loss function with respect to both the number of iterations and CPU time—first after 100 iterations (see \ref{fig:cvge1_mnVSnonm0}), and then separately for a fixed, predefined CPU time (see \ref{fig:cvge1_mnVSnonm01}).
From \cref{fig:cvge1_mnVSnonm0,fig:cvge1_mnVSnonm01}, we observe that the method performs consistently across different values of $u$, with all variants exhibiting comparable performance. Notably, the choice $u = 0.5$ yields slightly better results in terms of convergence speed, though the differences remain relatively minor.

\begin{figure}[t]
    \centering
    \begin{minipage}{\textwidth}
        \centering
        \subfigure[Gisette]{%
            \includegraphics[width=0.23\textwidth]{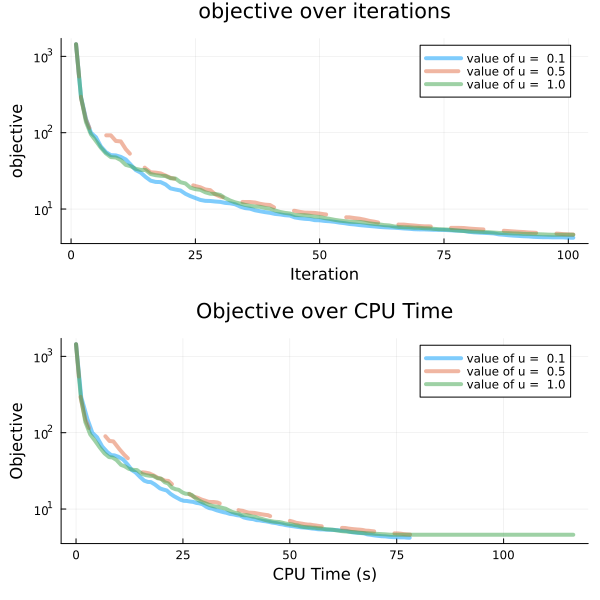}
            \label{fig:gisette1}
        }\hfill
        \subfigure[Colon-cancer]{%
            \includegraphics[width=0.23\textwidth]{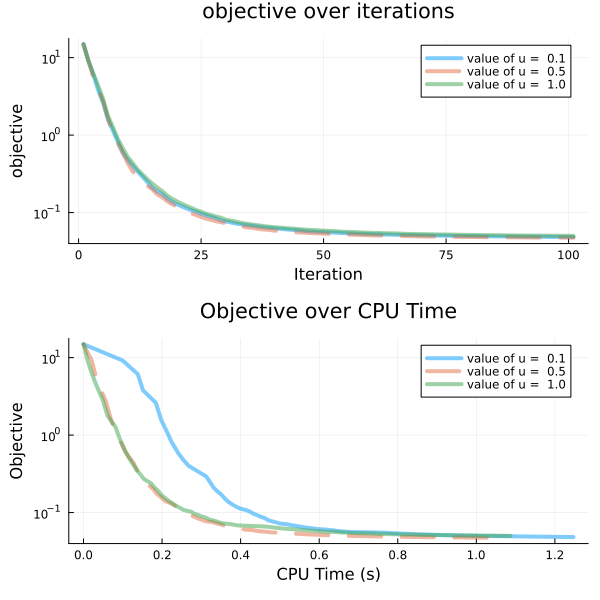}
            \label{fig:colon1}
        }\hfill
        \subfigure[Duke]{%
            \includegraphics[width=0.23\textwidth]{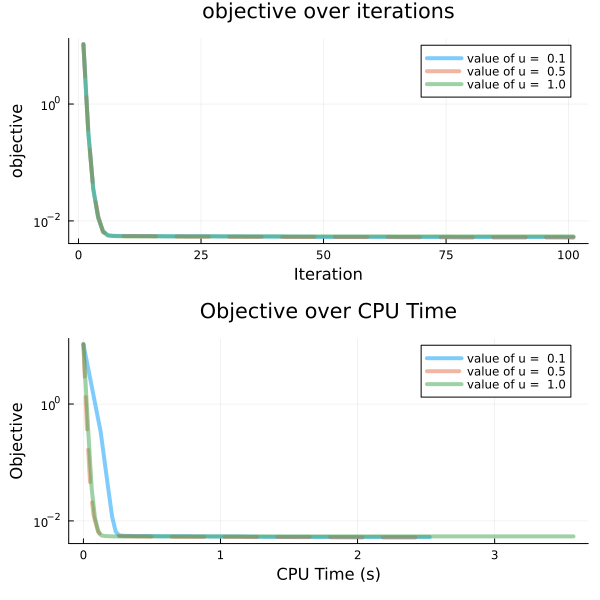}
            \label{fig:duke1}
        }\hfill
        \subfigure[Leukemia]{%
            \includegraphics[width=0.23\textwidth]{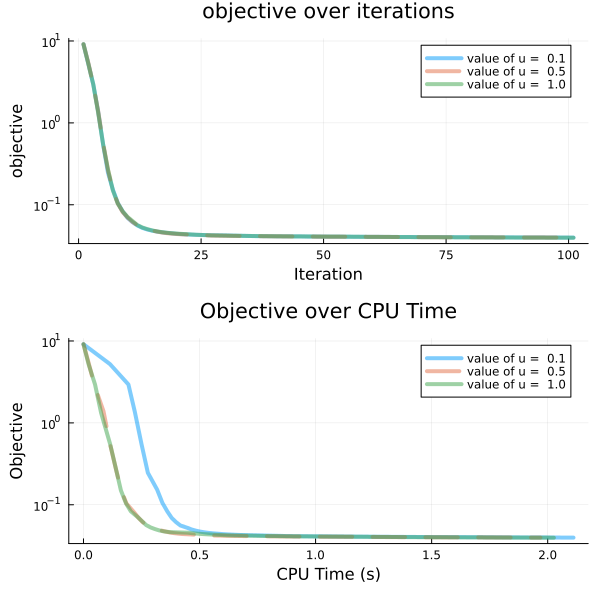}
            \label{fig:leukemia1}
        }
    \end{minipage}%
    \caption{Performance of Nonmonotone LiBCoD \ref{alg:nonmonotone-libcod} for different values of $u_k = \{ 0.1, 0.5, 1 \}$, with $\lambda = 0.001$, using $10\%$ block size for $100$ iterations on formulation \eqref{nonlinear1}. All variants show comparable performance.}
    \label{fig:cvge1_mnVSnonm0}
\end{figure}
\begin{figure}[t]
    \centering
    \begin{minipage}{\textwidth}
        \centering
        \subfigure[Gisette]{%
            \includegraphics[width=0.23\textwidth]{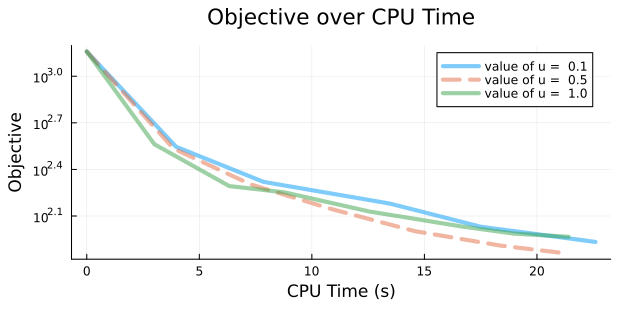}
            \label{fig:gisette11}
        }\hfill
        \subfigure[Colon-cancer]{%
            \includegraphics[width=0.23\textwidth]{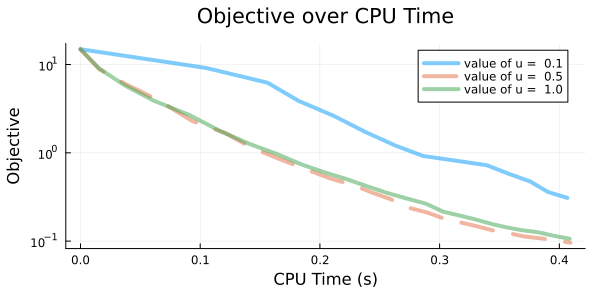}
            \label{fig:colon11}
        }\hfill
        \subfigure[Duke]{%
            \includegraphics[width=0.23\textwidth]{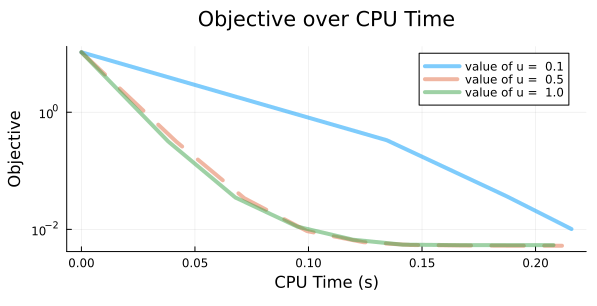}
            \label{fig:duke11}
        }\hfill
        \subfigure[Leukemia]{%
            \includegraphics[width=0.23\textwidth]{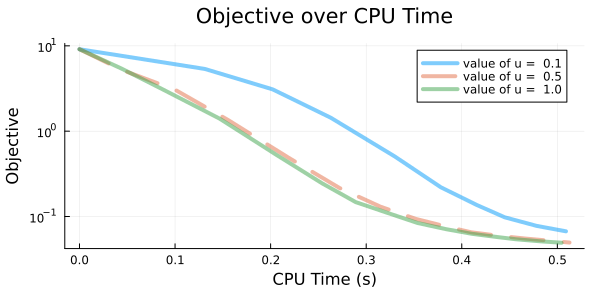}
            \label{fig:leukemia11}
        }
    \end{minipage}%
    \caption{Performance of Nonmonotone LiBCoD \ref{alg:nonmonotone-libcod} for different values of $u_k = \{ 0.1, 0.5, 1 \}$, with $\lambda = 0.001$, using $10\%$ block size under predefined CPU time on formulation \eqref{nonlinear1}. All variants show comparable performance.}
    \label{fig:cvge1_mnVSnonm01}
\end{figure}



\section{Conclusion}

In this paper, we proposed a block coordinate descent method with two variants, monotone and nonmonotone, for solving composite optimization problems. In both variants, we linearized the smooth part of the objective function in a Gauss-Newton approach and added a regularization term. By dynamically adjusting the regularization parameter, we established iteration complexity guarantees for achieving an \(\epsilon\)-stationary point of the problem. Additionally, our numerical experiments demonstrate the efficiency of the proposed method through comparisons with several baseline methods from the literature. While our results assume the separability of the nonsmooth part, it would be interesting to analyze the case where the nonsmooth part is not separable, as this would allow tackling a broader class of problems.

 \section*{Acknowledgement}
 YN and TV are supported by the Academy of Finland grant 345486. LE.B acknowledges funding from the European Union’s Horizon 2020 research and innovation programme under the Marie Skłodowska-Curie Grant Agreement No. 953348. 



\medskip



\end{document}